\documentclass[a4paper, 10pt]{amsart}
\input cyracc.def

\usepackage{amscd, amsmath,amssymb,amsthm, bm,comment,enumitem,eucal,url,mathrsfs,mathtools,MnSymbol,stmaryrd}
\usepackage[dvipdfmx]{hyperref}
\usepackage{graphicx}
\usepackage[all]{xy}
\usepackage{tikz}
\usetikzlibrary{cd}
\usetikzlibrary{patterns}

\DeclareMathOperator{\adm}{adm}

\DeclareMathOperator{\Ann}{Ann}

\DeclareMathOperator{\BT}{BT}

\DeclareMathOperator{\DD}{DD}

\DeclareMathOperator{\diag}{diag}
\DeclareMathOperator{\DS}{DS}
\DeclareMathOperator{\End}{End}

\DeclareMathOperator{\fingen}{fg}

\DeclareMathOperator{\Hom}{Hom}

\DeclareMathOperator{\id}{id}

\DeclareMathOperator{\Ind}{Ind}

\DeclareMathOperator{\Index}{Index}
\DeclareMathOperator{\Ker}{Ker}

\DeclareMathOperator{\cmod}{-mod}
\DeclareMathOperator{\op}{op}
\DeclareMathOperator{\pr}{pr}

\DeclareMathOperator{\Res}{Res}
\DeclareMathOperator{\sep}{sep}
\DeclareMathOperator{\Simple}{Simple}

\DeclareMathOperator{\Zf}{Zf}

\DeclareMathOperator{\Coh}{Coh}

\DeclareMathOperator{\Rep}{Rep}

\DeclareMathOperator{\GL}{GL}

\DeclareMathOperator{\Oo}{O}

\newcommand{\bC}{\mathbb{C}}

\newcommand{\bH}{\mathbb{H}}

\newcommand{\bQ}{\mathbb{Q}}
\newcommand{\bR}{\mathbb{R}}
\newcommand{\bZ}{\mathbb{Z}}
\newcommand{\cA}{\mathcal{A}}
\newcommand{\cB}{\mathcal{B}}
\newcommand{\cC}{\mathcal{C}}
\newcommand{\cD}{\mathcal{D}}

\newcommand{\cL}{\mathcal{L}}
\newcommand{\cM}{\mathcal{M}}
\newcommand{\cN}{\mathcal{N}}

\newcommand{\cO}{\mathcal{O}}

\newcommand{\fa}{\mathfrak{a}}

\newcommand{\fg}{\mathfrak{g}}

\newcommand{\fh}{\mathfrak{h}}

\newcommand{\fk}{\mathfrak{k}}
\newcommand{\fl}{\mathfrak{l}}

\newcommand{\fp}{\mathfrak{p}}

\def\Dbar{\leavevmode\lower.6ex\hbox to 0pt{\hskip-.23ex \accent"16\hss}D}
\theoremstyle{plain}
\newtheorem{thm}{Theorem}[subsection]
\newtheorem{defn-prop}[thm]{Definition-Proposition}
\newtheorem{cor}[thm]{Corollary}
\newtheorem{lem}[thm]{Lemma}
\newtheorem{prop}[thm]{Proposition}

\theoremstyle{definition}

\newtheorem{cons}[thm]{Construction}
\newtheorem{conv}[thm]{Convention}
\newtheorem{defn}[thm]{Definition}
\newtheorem{ex}[thm]{Example}

\newtheorem{rem}[thm]{Remark}
\begin{document}
	\title{Rationality patterns}
	\author{Takuma Hayashi}
	\address{Department of Mathematics, Graduate School of Science,
		Osaka Metropolitan University,
		3-3-138 Sugimoto, Sumiyoshi-ku Osaka 558-8585, Japan}
	\email{takuma.hayashi.forwork@gmail.com}
	\date{}
	
	
	\begin{abstract}
		In this paper, we establish general categorical frameworks that extend Loewy's classification scheme for finite-dimensional real irreducible representations of groups and Borel--Tits' criterion for the existence of rational forms of representations of $\bar{F}\otimes_F G$ for a connected reductive algebraic group $G$ over a field $F$ of characteristic zero and its algebraic closure $\bar{F}$. We also discuss applications of these general formalisms to the theory of Harish-Chandra modules, specifically to classify irreducible Harish-Chandra modules over fields $F$ of characteristic zero and to identify smaller fields of definition of irreducible Harish-Chandra modules over $\bar{F}$, particularly in the case of cohomological irreducible essentially unitarizable modules.
		 
	\end{abstract}
	
	\maketitle
	
	\tableofcontents
	\section{Introduction}\label{sec:intro}
	
	\subsection{Background from number theory}\label{sec:background}
	
	For an extension $F'/F$ of fields, rationality problems involve relating theories or objects over $F$ and $F'$, and studying them through their connections. A central question in these problems is determining a smaller field of definition for an object defined over a larger field. This naturally arises in various areas of number theory.
	
	Among these, the rationality of special values of automorphic $L$-functions is a long-standing problem. Roughly speaking, it concerns the assertion that ratios of special $L$-values to certain periods and powers of $2\pi\sqrt{-1}$ lie in specific number fields.
	
	A standard approach to the rationality of special $L$-values involves identifying rational structures on the space of automorphic forms, on (cohomological cuspidal) automorphic representations, or on (subspaces of) certain cohomology groups. Since this is a global problem, rationality at the infinite place is one of the ingredients in a natural definition of periods. From the representation-theoretic perspective, cohomological, essentially unitarizable, irreducible Harish-Chandra modules and their cohomology play a crucial role (see \cite{MR2439563} and \cite[Section~8]{MR3770183}, for example).

	\subsection{Rationality problems in representation theory}
	
	In the historical development of rationality problems in representation theory, we are naturally led to Loewy's classification scheme, which describes finite-dimensional real irreducible representations of a group $G$ in terms of finite-dimensional complex irreducible representations of $G$ and their behavior under complex conjugation:

	\begin{thm}[\cite{MR1500635}]\label{thm:loewy_original_1}
		Let $G$ be a group.
		\begin{enumerate}
			\item We have a well-defined bijection between the set of complex conjugacy classes of isomorphism classes of finite-dimensional complex irreducible representations $V$ of $G$ and the set of isomorphism classes of finite-dimensional real irreducible representations of $G$. This bijection is given by assigning to each such $V$ any irreducible subrepresentation $V_\bR$ of the restriction of scalars to the real numbers $\Res_{\bC/\bR} V$.
			\item For a finite-dimensional complex irreducible representation $V$ of $G$, the real representation $\Res_{\bC/\bR} V$ is reducible if and only if $V$ admits a real form.
		\end{enumerate}
	\end{thm}

	To establish a numerical criterion for the existence of real forms, one can define a sign $\Index(V) \in \{\pm 1\}$ for a self-conjugate finite-dimensional complex irreducible representation $V$ (\cite[Section~VIII]{E1914}). This sign is referred to as the index of $V$. As explained below, this is related to the division algebra of the endomorphisms of $V_\bR$.

	\begin{thm}[\cite{MR1500635}]\label{thm:loewy_original_2}
		Let $V$ be a finite-dimensional complex irreducible representation of a group $G$. Then:
		\begin{enumerate}
		\item $V$ admits a real form if and only if $V$ is self-conjugate with $\Index(V) = 1$.
		\item The division algebra $\End_G(V_\bR)$ of endomorphisms of $V_\bR$ is isomorphic to $\bR$ (resp.~$\bC$, $\bH$) if $V$ is self-conjugate with $\Index(V) = 1$ (resp.~$V$ is not self-conjugate, $V$ is self-conjugate with $\Index(V) = -1$). For the definitions of $\bR$, $\bC$, and $\bH$, see the notation section.
		\end{enumerate}
		
	\end{thm}
	
	Subsequently, in \cite[Section 12]{MR207712}, Borel--Tits generalized the notion of indices by introducing a second Galois cohomology class, denoted $\beta_V$, associated with a connected semisimple (or more generally, reductive) algebraic group $G$ over a field $F$ of characteristic zero and a self-conjugate irreducible representation $V$ of $\bar{F} \otimes_F G$, where $\bar{F}$ is an algebraic closure of $F$. This class serves as an obstruction to the existence of $F$-forms (\cite[12.8. Proposition]{MR207712}). For $V$ not self-conjugate, Tits defined $\beta_V$ by replacing $F$ with the field of rationality of $V$ (\cite[7.1. Notations]{MR277536}). Numerical aspects of $\beta_V$ are developed in \cite[Sections 2.2 and 2.3]{MR4627704} and \cite{hayashierror}, based on the Borel--Weil theorem.

	In a related development, Tits established a bijection analogous to Theorem \ref{thm:loewy_original_1} in the setting of connected semisimple algebraic groups above (\cite[7.2. Th\'eor\`eme]{MR277536}). He also verified an analog of Theorem \ref{thm:loewy_original_2} using Serre's canonical isomorphism between the second Galois cohomology group and the Brauer group of $F$ (\cite[3.2. Lemme]{MR277536}).

	In \cite{MR3770183}, Januszewski brought the insight of Theorem \ref{thm:loewy_original_2} (1) into the theory of $(\fg,K)$-modules to give its analog for quadratic extensions. The index is called the Frobenius--Schur indicator in his paper, following their work on a formula for the indices in terms of characters of complex irreducible representations of a finite group $G$ in \cite{zbMATH02646565}.
	
	For computational aspects, Januszewski verified the local-global principle of the index (\cite[Theorem 5.5]{MR3770183}). This motivates us to develop the theory of $(\fg,K)$-modules over non-Archimedean local fields of characteristic zero, which we do not address in this paper but which could be an interesting topic for future work.

	\subsection{Towards a categorical and axiomatic approach}\label{sec:towards}
	
	One can readily see that the arguments of Loewy, Borel--Tits, and Tits are applicable in broader contexts. Indeed, the core idea underlying their work is the analysis of descent data. In particular, Loewy's and Tits' results may be interpreted, respectively, as the classification of simple descent data and their identification with real irreducible representations, and with irreducible representations over $F$, via Galois descent. For modern expositions in the setting of real Lie algebras, following Cartan \cite{E1914}, see \cite{MR102534} and \cite[\S8]{MR2041548}.

	Leaving aside the issue of continuity in the infinite setting, descent data (or systems) can be formulated whenever one has a category equipped with an action of a group $\Gamma$. This leads us to the notion of a $\Gamma$-category. Descent systems were introduced in this general context (see \cite{MR0711065,MR0861356,MR1126178}). In dealing with infinite Galois extensions, we must take into account the continuity of descent systems, from the perspective of classical Galois descent in algebra and algebraic geometry.

	\subsection{Results of this paper, I}\label{sec:results}
	
	As observed in the previous section, $\Gamma$-categories and descent systems serve as a basis for a generalization of the works of \cite{MR1500635,MR207712,MR277536} on rationality problems. Based on this idea, we introduce general categorical frameworks in which the arguments of Loewy, Borel--Tits, and Tits can be applied.
	
	The simplest way to generalize the bijection in Theorem \ref{thm:loewy_original_1} (1) is to consider a full subcategory $\DD(\cC)$ of the category $\DS(\cC)$ of descent systems in a $\Gamma$-category $\cC$. To obtain a simple object of $\cC$ from a simple descent datum, let us note a key condition that every object of $\cC$ is assumed to be of finite length. This setting enables us to carry out arguments similar to those in the classical case:

	\begin{thm}[Theorem \ref{thm:Loewy}]\label{thm:Loewy_intro}
		Let $(\cC, \DD(\cC))$ be a Loewy pair (Definition \ref{def:lp} (1)). Then, there is a bijection between the sets of isomorphism classes of simple objects in $\DD(\cC)$ and the Galois conjugacy classes of isomorphism classes of simple objects in $\cC$.
	\end{thm}
	
	We refer to the objects in $\DD(\cC)$ as descent data. When working with a finite Galois extension, the category $\DS(\cC)$ provides a typical example of $\DD(\cC)$ (see Definition \ref{def:lp} (2) and Definition \ref{defn:descent_system}).
	
	Note that Theorem \ref{thm:Loewy_intro} recovers Theorem \ref{thm:loewy_original_1} (2) as a special instance. More generally, a similar statement holds true for Loewy pairs $(\cC,\DS(\cC))$ with respect to quadratic Galois extensions (Propositions \ref{prop:res} and \ref{prop:Loewy}).

	To obtain $\DD(\cC)$ as the union of descent systems over finite intermediate Galois extensions $F''/F$ in $F'$, we introduce the notion of a Loewy datum. A Loewy datum consists of certain $F''$-linear abelian categories---referred to as reference categories---equipped with actions of the corresponding Galois groups and compatible base change functors (Definition \ref{defn:Ldatum}). In this framework, the bijection in Theorem \ref{thm:Loewy_intro} is obtained objectwise from the bijections for the reference categories (see Theorem \ref{thm:basic_Loewy_datum}). For instance, categories of finite-dimensional representations of affine group schemes over fields provide typical examples of Loewy data (Example \ref{ex:repG}).

	We next construct the Borel--Tits cocycle $\beta^{\BT}_M$ for a (self-conjugate) simple object $M$. In the finite Galois case, we define $\beta^{\BT}_M$ in a canonical way (Definition-Proposition \ref{def-prop:bt_finite}). This definition agrees with the facile one (Proposition \ref{prop:key_computation_BT}). In the infinite Galois case, we reduce to the finite case by taking an $F''$-form of $M$ for some finite Galois extension $F''/F$ within the given infinite Galois extension $F'/F$.

	We introduce the notion of Borel--Tits data, which makes this construction possible (Definition \ref{defn:BTdatum}). The main result for them is:

	\begin{thm}[Definition-Proposition \ref{defprop:BT}]\label{thm:BT_intro}
		Suppose that we are given a Borel--Tits datum on a category $\cC$. Let $M \in \cC$ be a self-conjugate simple object whose endomorphisms are only scalars. Then, we have an obstruction class $\beta^{\BT}_M$ to the existence of a descent datum on $M$.
	\end{thm}
	
	\begin{ex}[Example \ref{ex:affinegroupscheme}, Proposition \ref{prop:key_computation_BT}]\label{ex:comparison}
		Let $G$ be an affine group scheme over a field $F$, and $F'$ be a possibly infinite Galois extension of $F$. Then the categories of representations of $F''\otimes_F G$ form a Borel--Tits datum, where $F''$ runs through $F'$ and finite Galois extensions of $F$ in $F'$. Moreover, the Borel--Tits cocycles $\beta^{\BT}_V$ in this paper agrees with the original one $\beta_V$ in \cite[Section~12]{MR207712}.
	\end{ex}
	
	Let us also consider its converse, i.e., we discuss whether $F' \otimes_F M$ is simple for $M \in \cC_F$ simple. This condition is known as one of the possible definitions of pseudo-absolute simplicity when $F'$ is a separable closure of $F$ (\cite[Definition-Proposition 3.30]{hayashisuper}). If $F$ is perfect and $F'$ is algebraically closed, this is one of the possible definitions of absolute simplicity, which we address later. We verify that some possible definitions of pseudo-absolute simplicity are equivalent:

	\begin{thm}[Theorem \ref{thm:abs_simple}]\label{thm:abs_simple_intro}
	Let $\cC$ be an $F'$-linear abelian category with the structure of a strong Borel--Tits datum. Then, for a simple object $M$ of $\cC_F$, the following conditions are equivalent:
	\begin{enumerate}
		\renewcommand{\labelenumi}{(\alph{enumi})}
		\item $F' \otimes_F M$ is simple,
		\item $F'' \otimes_F M$ is simple for every finite Galois extension $F''/F$ in $F'$, and
		\item $M$ only has scalar endomorphisms.
	\end{enumerate}
	
	\end{thm}

	We also introduce the notion of (strong) Loewy--Borel--Tits data to generalize Theorem \ref{thm:loewy_original_2} (Definition \ref{defn:LBT}). Roughly speaking, a Loewy--Borel--Tits datum is a combination of Loewy and Borel--Tits data. The strong condition requires that simple objects have only scalar endomorphisms. We give an algebraic interpretation of $\beta^{\BT}_M$:

	\begin{thm}[Corollary \ref{cor:LBT}]\label{thm:LBT_intro}
		Suppose that we are given a strong Loewy--Borel--Tits datum on a category $\cC$. Let $M$ be a simple object of $\cC$, and $S$ be the corresponding simple descent datum. Then:
		\begin{enumerate}
			\item $M$ is self-conjugate if and only if the division algebra $\End_{\DD(\cC)}(S)$ of endomorphisms of $S$ is central.
			\item If the equivalent conditions of (1) are satisfied, then $(\beta^{\BT}_M)^{-1}$ represents the similarity class of $\End_{\DD(\cC)}(S)$ in the Brauer group of $F'/F$.
		\end{enumerate}
	\end{thm}
	
	This is immediate from our definition of Borel--Tits cocycles if $\Gamma$ is finite. The infinite case is obtained by building up the results of Theorems \ref{thm:Loewy_intro} and \ref{thm:LBT_intro} for the finite case.

	\subsection{Towards applications to $(\fg,K)$-modules}\label{sec:apptoHC}
	
	As explained in Section~\ref{sec:background}, the rationality of Harish-Chandra modules is a fundamental problem in number theory. To generalize the results of \cite{MR3770183}, we plan to apply the results from Section~\ref{sec:results}. For this purpose, we use the base change functor for $(\fg,K)$-modules introduced in \cite{MR3770183} (see also \cite{MR3853058}). 
	
	In order to obtain a Loewy--Borel--Tits datum for certain familiar categories of $(\fg,K)$-modules, a key step is to verify Harish-Chandra's finiteness theorem, which provides equivalent conditions for the finite-length property in the present algebraic setting (see, for example, \cite[4.2.1. Theorem]{MR929683} and \cite[Corollary 7.207]{MR1330919}). These equivalences are stated in \cite[Sections 3.1 and 3.2]{MR3770183}. 
	
	However, the proof in \cite{MR3770183} reduces the argument to the complex case by embedding the ground field into $\bC$. To the best of the author's knowledge, existing proofs of Harish-Chandra's finiteness theorem over $\bC$ in the literature are analytic. For our purposes, however, an algebraic proof valid in more general settings is required. We resolve it in this paper.
	
	To establish the local-global principle for Borel--Tits cocycles, it is essential that irreducibility is preserved under base change of algebraically closed fields. This is a special case of Januszewski's notion of absolute irreducibility, as defined in \cite[Definition 3.3]{MR3770183}. At first glance, this condition appears quite strong. 
	
	On the other hand, we have introduced the notion of pseudo-absolute simplicity (Theorem \ref{thm:abs_simple_intro}). A fundamental question is whether these two notions are equivalent. In fact, various notions of absolute simplicity (or irreducibility) appear across different branches of algebra. They are often known to be equivalent by virtue of Jacobson's density theorem.

	For applications to cohomological, essentially unitarizable, irreducible Harish-Chandra modules, we note that such modules can be realized as cohomologically induced modules (\cite[Theorem 5.6]{MR762307}; see also \cite[Sections 5.2 and 5.4]{hayashijanuszewski} for the disconnected case). In particular, rational forms can be obtained by choosing rational forms of the corresponding parabolic subalgebras and characters, since the cohomological induction functor satisfies the base change property (\cite[Theorems D and H]{MR3853058}). Further aspects of the self-conjugacy property can be explored through the explicit construction.

	In preparation for the computation of the Borel--Tits cocycles, we note a result of Januszewski concerning indices:
	
	\begin{prop}[{\cite[Proposition 5.8]{MR3770183}}]\label{prop:Ktype_intro}
		Let $(\fg,K)$ be a Harish-Chandra pair $(\fg,K)$ over a field $F$ of characteristic zero with $\fg$, $K$ reductive and $(\fg,\fk)$ symmetric, where $\fk$ is the Lie algebra of $K$.
		
		Let $F'/F$ be a quadratic extension of fields. Let $X$ be a self-conjugate absolutely irreducible Harish-Chandra module, and $\tau$ be a self-conjugate $F'\otimes_FK$-type with multiplicity one in $X$. Then the indices of $X$ and $\tau$ are equal.
	\end{prop}

	\subsection{Results of this paper, II}\label{sec:resultsII}
	
	For simplicity of presentation, we consider a Harish-Chandra pair $(\fg,K)$ over $F$ in Proposition \ref{prop:Ktype_intro}.
	
	The main result of Section \ref{sec:(g,K)mod} is:
	
	\begin{thm}[Theorem \ref{thm:hc_setting}, Proposition \ref{prop:fl->Zf}]\label{thm:main}
		For a Galois extension $F'/F$, $(F'' \otimes_F \fg, F'' \otimes_F K)$-modules of finite length form a Loewy--Borel--Tits datum, where $F''$ runs through $F'$ and finite Galois extensions of $F$ in $F'$.
	\end{thm}
	
	The essential step is to verify that the functor $F' \otimes_F (-)$ preserves the finite length property. This is accomplished by providing an algebraic—and partially algebro-geometric—proof of Harish-Chandra's finiteness theorem in Section~\ref{sec:rem}.
	
	Note that in Theorem \ref{thm:hc_setting}, we consider admissible modules of finite length in an attempt to remove the symmetric assumption of $(\fg, \fk)$ (see Definition \ref{defn:HCmod} for the definition of admissible $(\fg,K)$-modules). When the symmetric assumption is in place, the $(\fg, K)$-modules of finite length are admissible by Harish-Chandra's finiteness theorem (Proposition \ref{prop:fl->Zf}), but without it, they may fail to be admissible.

	As a consequence, we obtain the notion of Borel--Tits cocycles when $F' = \bar{F}$. Recall that these cocycles represent obstruction classes to the existence of $F$-forms. As an application, we can determine minimal fields of definition:

	\begin{cor}[Definition \ref{defn:fld_rat}, Proposition \ref{prop:min_field}, Remark \ref{rem:existence}]\label{cor:min}
		Let $(\fg, K)$ be as before, and $V$ be an irreducible $(\bar{F} \otimes_F \fg, \bar{F} \otimes_F K)$-module. Let $\Gamma_F$ be the absolute Galois group of $F$. Set
		$\Gamma_V = \{\sigma \in \Gamma_F : \, {}^\sigma V \cong V \}$,
		where ${}^\sigma V$ is the $\sigma$-twist of $V$. Let $F(V) \subset \bar{F}$ be the fixed point subfield by $\Gamma_V$. Then:
		\begin{enumerate}
			\item $\Gamma_V$ is open in $\Gamma_F$. In particular, $F(V)$ is finite over $F$.
			\item Let $F'/F(V)$ be a finite extension in $\bar{F}$ with the property that the image of $\beta^{\BT}_M \in H^2 (\Gamma_V, \bar{F}^\times)$ in $H^2(\Gamma_{\bar{F}/F'}, \bar{F}^\times)$ is trivial. Then $V$ admits an $F'$-form.
			\item The field $F'$ in (2) does exist.
			\item If $V$ admits an $F'$-form for a finite extension $F'/F$ in $\bar{F}$, then we have $F(M) \subset F'$.
		\end{enumerate}
		
	\end{cor}
	
	In fact, it follows from (4) and Theorem~\ref{thm:BT_intro} that every field of definition arises in this way. To obtain a minimal field of definition, it suffices to find a minimal $F'$ among those appearing in (2).
	
	For the analysis of the Borel--Tits cocycles, we verify the following:

	\begin{thm}[Theorem \ref{thm:abs_irr_hc}]\label{thm:mainB}
		For an irreducible $(\fg,K)$-module $V$, $F' \otimes_F V$ is irreducible for every extension $F'/F$ of fields if and only if $\bar{F} \otimes_F V$ is irreducible.
	\end{thm}
	
	For the proof, we may assume that $F = \bar{F}$. The key fact is that the Hom space of each (nonzero) $K$-isotypic component of $V$ is simple as a module over the $K$-invariant part $U(\fg)^K$ of the universal enveloping algebra of~$\fg$ (\cite[3.5.4. Proposition]{MR929683}). The assertion then easily follows from applying Jacobson's density theorem to that Hom space.

	As a consequence, we deduce the local-global principle for Borel--Tits cocycles:

\begin{cor}[Proposition \ref{prop:lg}]
	Assume that $F$ is a number field. Let $V$ be a self-conjugate irreducible $(\bar{F} \otimes_F \fg, \bar{F} \otimes_F K)$-module. For each place $v$ of $F$, we fix an embedding $\bar{F} \hookrightarrow \bar{F}_v$, where $F_v$ is the completion of $F$ at $v$.
	\begin{enumerate}
		\item For each place $v$, $\bar{F}_v \otimes_{\bar{F}} V$ is irreducible and self-conjugate. 
		\item The cocycle $\beta^{\BT}_V$ is trivial if and only if so are $\beta^{\BT}_{\bar{F}_v \otimes_{\bar{F}} V}$ for all places $v$ of $F$.
	\end{enumerate}
	\end{cor}

	We further generalize Proposition \ref{prop:Ktype_intro}, establishing a relation between the Borel--Tits cocycles of Harish-Chandra modules and those of certain $K$-types:

	\begin{prop}[Proposition \ref{prop:mult_one}]\label{prop:mult-one}
		Let $V$ be an irreducible $(\bar{F} \otimes_F \fg, \bar{F} \otimes_F K)$-module, and $\tau$ be a self-conjugate $\bar{F} \otimes_F K$-type of multiplicity one in $V$. Then we have $\beta^{\BT}_{V} = \beta^{\BT}_{\tau}$.
		
	\end{prop}
	
	We conclude this paper with applications to cohomological, essentially unitarizable, irreducible Harish-Chandra modules:

	\begin{thm}[Propositions \ref{prop:form}, \ref{prop:rat_coh}, \ref{prop:BT}]
		Let $F$ be a subfield of $\bR$, and let $G$ be a connected reductive algebraic group over $F$ equipped with an involution $\theta$, with associated symmetric subgroup $K$. Assume that $\bR \otimes_F \theta$ is a Cartan involution. Let $\bar{F}$ denote the algebraic closure of $F$ in $\bC \supset \bR$.
		
		Let $V$ be an irreducible representation of $\bC \otimes_F G$ with regular highest weight. Let $X$ be a cohomological, essentially unitarizable, irreducible $(\bC \otimes_F \fg, \bC \otimes_F K)$-module with coefficient $V$, where $\fg$ is the Lie algebra of $G$.
		
		\begin{enumerate}
			\item The Harish-Chandra module $X$ admits a unique $\bar{F}$-form $X_{\bar{F}}$, up to isomorphism.
			\item Assume that $X_{\bar{F}}$ is self-conjugate. Then its Borel--Tits cocycle is determined by that of the $\bar{F}$-form $X_{\min, \bar{F}}$ of the minimal $\bC \otimes_F K$-type of $X$.
		\end{enumerate}
	\end{thm}
	
	In fact, part (1) is an easy consequence of the Vogan--Zuckerman theory \cite{MR762307} (see Section \ref{sec:apptoHC}).
	
	Part (2) follows from Proposition \ref{prop:mult-one}. The regularity assumption on $V$ ensures that $X$ admits a unique minimal $\bC \otimes_F K$-type, which appears with multiplicity one.
	
	To analyze self-conjugacy, we consider the Galois twists of the corresponding cohomological induction over $\bar{F}$. The field of rationality $F(X_{\bar{F}})$ coincides with the field of definition of $X_{\bar{F}}$ (see Proposition \ref{prop:rat_coh} (2)). Even without the regularity assumption on $V$, this argument allows us to estimate $F(X_{\bar{F}})$ (Proposition \ref{prop:rat_coh}).
	
	Finally, we note two ways to compute the Borel--Tits cocycle $\beta^{\BT}_{X_{\min, \bar{F}}}$ of $X_{\min, \bar{F}}$ over the field of rationality of the cohomological induction. Notably, this cocycle is easier to compute than $\beta^{\BT}_{X_{\bar{F}}}$.
	
	It follows from \cite[THEOREM 5.3]{MR762307} (or \cite[Theorem 5.4.7]{hayashijanuszewski} for the disconnected case) that $X_{\min, \bar{F}}$ is the induction of a certain character $\lambda^\sharp$ of $Q^-\cap \bar{F}\otimes_F K$ to $\bar{F}\otimes_F K$, where $Q^-$ is an $\bar{F}\otimes_F \theta$-stable parabolic subgroup of $\bar{F}\otimes_F G$. Consequently, $\beta^{\BT}_{X_{\min, \bar{F}}}$ can be expressed analogously to \cite[Lemma 2.2.4]{MR4627704} and \cite{hayashierror} (Proposition \ref{prop:BT}). This gives the first of the two methods of computation. We omit the precise formulation here since the explicit cocycle requires additional care due to continuity arguments (cf.~\cite{hayashierror}).

	We next state the other way:
	
	\begin{prop}[Proposition \ref{prop:K^0}]
		Let $Q^-$ be as above. Assume that $Q^-\cap (\bar{F}\otimes_F K)$ meets every component of $\bar{F}\otimes_F K$. Then, $X_{\min, \bar{F}}$ is irreducible as a representation of $\bar{F} \otimes_F K^0$, where $K^0$ denotes the identity component of $K$. Moreover, $\beta^{\BT}_{X_{\min, \bar{F}}}$ coincides with the Borel--Tits cocycle of the representation $X_{\min, \bar{F}}$ of $\bar{F} \otimes_F K^0$.
	\end{prop}
	
	This provides an alternative method for computing the Borel--Tits cocycle. Indeed, results of \cite{MR277536,MR4627704,hayashierror} may be applied to $K^0$ in concrete examples.

	\subsection{Future direction}
	
	Harris proposed the study of twisted D-modules on flag varieties over fields $F$ of characteristic zero as a method to construct and classify rational forms of Harish-Chandra modules (\cite[Section~2]{MR3053412}). Januszewski observed that this task involves addressing rationality problems of related geometric objects. To resolve this issue, Harris amended his work by considering a finite extension of the base field (\cite{MR4073199}).
	
	In a forthcoming work, we will demonstrate that our general theories for rationality problems can be applied to equivariant holonomic twisted D-modules:

	\begin{thm}
		Let $X$ be a smooth algebraic variety over a field $F$ of characteristic zero, equipped with a linear algebraic group $K$. Let $\cA$ be a $K$-equivariant tdo over $X$. Then the categories of $(F' \otimes_F K)$-equivariant holonomic left $(F' \otimes_F \cA)$-modules form a strong Loewy--Borel--Tits datum, where $F'$ runs through an algebraic closure $\bar{F}$ of $F$ and finite Galois extensions of $F$ in $\bar{F}$.
		
	\end{thm}
	
	As an application, we will revisit Harris' rationality results: We follow the idea of Corollary \ref{cor:min} to provide a general method for finding minimal fields of definition of equivariant holonomic twisted D-modules on flag varieties.

	\section{Organization of this paper}

	In Section~\ref{sec:lbt}, we establish general categorical frameworks for Loewy's classification scheme and Borel--Tits' criterion. We review the notions of $\Gamma$-categories and descent systems in Section~\ref{sec:gamma_cat}. Then, we formulate Loewy pairs in Section~\ref{sec:Loewy_pair}. For constructing examples of Loewy pairs, we also introduce the notion of Loewy data for a Galois extension $F'/F$ of fields. 
	
	In Section~\ref{sec:Loewy}, we prove Theorem \ref{thm:Loewy_intro}. For the computation of this bijection for Loewy data, we show that it is essentially reduced to the finite case (Theorem \ref{thm:basic_Loewy_datum}). We discuss this bijection in the case of quadratic Galois extensions to recover Loewy's original description of the bijection (Proposition \ref{prop:Loewy}). 
	
	Section~\ref{sec:BTcocycle} is devoted to the theory of Borel--Tits data and Borel--Tits cocycles. In particular, we verify Theorem \ref{thm:BT_intro}. As an example of the computation on the bijection in Theorem \ref{thm:Loewy_intro}, we relate the simplicity of the underlying object of a simple descent datum with the algebra of its endomorphisms at the ends of these two sections (Proposition \ref{prop:existence_DD}, Theorem \ref{thm:abs_simple}). They provide general frameworks for (pseudo-)absolute simplicity.
	
	In Section~\ref{sec:relation}, we introduce Loewy--Borel--Tits data to establish Theorem \ref{thm:LBT_intro}.

	In Section~\ref{sec:(g,K)mod}, we discuss applications of our rationality theory from Section~\ref{sec:lbt} to Harish-Chandra modules and provide further results. In Section~\ref{sec:goal}, we prove Theorem \ref{thm:main}. In Section~\ref{sec:abs_irr}, we show Theorem \ref{thm:mainB}. We discuss generalizations of \cite[Theorem 5.5 and Proposition 5.8]{MR3770183} and minimal fields of definition in Section~\ref{sec:more}. We make remarks on finiteness conditions in Section~\ref{sec:rem}. Finally, in Section~\ref{sec:coh}, we combine the Vogan--Zuckerman theory \cite{MR762307} with our results to discuss cohomological irreducible essentially unitarizable Harish-Chandra modules.

	\section{Notation}
	
	\subsection{Elementary concept}
	Let $\bZ$ be the ring of integers. Let $\bQ$ (resp.~$\bR$, $\bC$) denote the field of rational (resp.~real, complex) numbers. Let $\bH$ be the real division algebra of quaternions.
	
	For a number field $F$ and a place $v$, let $F_v$ denote the $v$-adic completion of $F$.
	
	For a group $\Gamma$, let $e$ denote its identity element. For a $\Gamma$-set $S$, we denote its fixed-point subset by $S^\Gamma$.
	
	For a field $F$, let $F^{\sep}$ (resp.~$\bar{F}$) denote a separable (resp.~algebraic) closure of $F$, which we fix unless specified otherwise. Let $F^\times$ denote the group of units of $F$.
	
	Let $F'/F$ be a possibly infinite Galois extension of fields. We denote its Galois group by $\Gamma_{F'/F}$. If $F' = F^{\sep}$, we omit $F'/$ and write $\Gamma_F$ instead. Throughout this paper, let $\Lambda_{F'/F}$ denote the set of finite Galois extensions of $F$ in $F'$.
	
	For a quadratic Galois extension, the induced Galois involutions on objects will be denoted by $\bar{}$.

	For a commutative ring $A$ and an $A$-module $M$, we denote its annihilator by $\Ann_A(M)$. For $a\in A$, we set $\Ann_M(a)\coloneqq\{m\in M:~am=0\}$.
	
	For a finite-dimensional vector space $V$ over a field $F$, let $F[V]$ denote the symmetric power of the dual of $V$, which will be regarded as the algebra of polynomial functions on $V$ in an appropriate sense.

	\subsection{Categories}
	
	For an object $X$ of a category, we write $\id_X$ for its identity map.
	
	For a category $\cC$, we write $\cC^{\op}$ for its opposite category. We denote the Hom sets of $\cC$ by $\Hom_{\cC}(-,-)$ (if $\cC$ is locally small). We write $\End_{\cC}(X) = \Hom_{\cC}(X, X)$ for $X \in \cC$. If $\cC$ is the category of modules over a ring $A$, we write $\Hom_{A}(-, -)$. We define $\End_A(-)$ in a similar way.
	
	For a natural transformation $\alpha: F \Rightarrow G$ of functors $F, G: \cC \rightrightarrows \cD$ and $X \in \cC$, we denote the corresponding morphism $F(X) \to G(X)$ by $\alpha_X$.

	If we are given natural transformations of functors depicted as
	\[
	\begin{tikzcd}[column sep=huge]
		\cC_1
		\arrow[bend left=50]{r}[name=U,label=above:$F_1$]{}
		\arrow[bend right=50]{r}[name=D,label=below:$G_1$]{} &
		\cC_2
		\arrow[shorten <=10pt,shorten >=10pt,Rightarrow,to path={(U) -- node[label=right:$\alpha$] {} (D)}]{}
		\arrow[bend left=50]{r}[name=U2,label=above:$F_2$]{}
		\arrow[bend right=50]{r}[name=D2,label=below:$G_2$]{}
		&\cC_3,\arrow[shorten <=10pt,shorten >=10pt,Rightarrow,to path={(U2) -- node[label=right:$\beta$] {} (D2)}]{}
	\end{tikzcd}
	\]
	we define a natural transformation $\beta \circ \alpha: F_2 \circ F_1 \Rightarrow G_2 \circ G_1$ by
	\[
	(\beta \circ \alpha)_X = \beta_{G_1(X)} \circ F_2(\alpha_X) = G_2(\alpha_X) \circ \beta_{F_1(X)}: F_2(F_1(X)) \to G_2(G_1(X))
	\]
	for $X \in \cC$ (the horizontal composition).
	If we are given natural transformations of functors depicted as
	\[
	\begin{tikzcd}[column sep=huge]
		\cC_1
		\arrow[bend left=50]{r}[name=U,label=above:$F$]{}
		\arrow[bend right=50]{r}[name=D,label=below:$H$]{} 
		\arrow{r}[name=M]{}\ar[r,"G" description]
		&\cC_2,\arrow[shorten <=2pt,shorten >=-1pt,Rightarrow,to path={(U) -- node[label=right:$\alpha$] {} (M)}]{}
		\arrow[shorten <=4pt,shorten >=-3pt,Rightarrow,to path={(M) -- node[label=right:$\beta$] {} (D)}]{}
	\end{tikzcd}
	\]
	we define a natural transformation $\beta \cdot \alpha : F \Rightarrow H$ by
	\[
	(\beta \cdot \alpha)_X = \beta_X \circ \alpha_X : F(X) \to H(X)
	\]
	for $X \in \cC$ (the vertical composition).
	
	For an essentially small abelian category $\cC$, let $\Simple(\cC)$ denote the set of isomorphism classes of simple objects in $\cC$.

	\subsection{Algebraic varieties}
	
	Let $X$ be an algebraic variety over a field $F$, equipped with an action of a linear algebraic group $G$. Then we denote the structure sheaf of $X$ by $\cO_X$. Let $\Coh(X, G)$ denote the category of $G$-equivariant coherent $\cO_X$-modules.

	\subsection{Lie algebras}

	Let $\fg$ be a Lie algebra over a field. Then the enveloping algebra of $\fg$ will be denoted by $U(\fg)$. We write $Z(\fg)$ for the center of $U(\fg)$.
	
	For an involution $\theta$ of a Lie algebra $\fg$ over a field of characteristic not two, let $\fg^{\theta}$ be the $\theta$-fixed point subset of $\fg$.
	
	For a reductive Lie algebra $\fg$ over a field of characteristic zero (\cite[Chapter I, \S 6, 4]{MR1728312}), let $Z_\fg$ and $\fg^{\mathrm{ss}}$ denote the center and the semisimple Lie subalgebra of $\fg$, respectively.
	
	Let $\fg$ be a reductive Lie algebra over an algebraically closed field, and $\fh$ be a Cartan subalgebra of $\fg$ (\cite[Chapter VII, \S2.1 DEFINITION 1]{MR2109105}). Fix a positive system of $(\fg,\fh)$ (see \cite[Chapter VIII, \S2.2 REMARK 4]{MR2109105} for the notion of roots of $(\fg,\fh)$ if necessary). Then for $\lambda\in\fh^\vee$, we define a character $\chi_\lambda:Z(\fg)\to\bC$ as in \cite[(5.41)]{MR1920389}.
	
	\subsection{Algebraic groups}

	For an affine group scheme $G$ over a field $F$, let $\Rep(G)$ be the category of finite-dimensional representations of $G$. For (possibly infinite-dimensional) representations $V$ and $W$ of $G$, we denote the space of homomorphisms from $V$ to $W$ as representations of $G$ by $\Hom_G(V, W)$. For a representation $V$ of $G$, we denote its dual as a representation by $V^c$. That is, we put the standard symmetric monoidal structure on the category of representations of $G$. Then $V^c$ is the internal Hom representation from $V$ to the trivial representation $F$.
	
	For representations $V$ and $\tau$ of a reductive algebraic group $K$ over a field $F$ of characteristic zero, with $\tau$ irreducible, let $V(\tau)$ denote the $\tau$-isotypic component of $V$. That is, $V(\tau)$ is the image of the injective map $\Hom_K(\tau, V) \otimes_F \tau \to V$ defined by evaluation.

	For a smooth linear algebraic group $G$, we denote its unit component by $G^0$. Write $Z_G$ for the center of $G$. Let $G^{\mathrm{ss}}$ denote the derived subgroup of $G$.
	
	For a linear algebraic group $G$ over a field of characteristic zero with an involution $\theta$, let $G^\theta$ be the $\theta$-fixed point subgroup of $G$.
	
	For a smooth linear algebraic group, we denote its Lie algebra by the corresponding German letter.
	
	For a representation $V$ of a linear algebraic group $G$, let $V^G$ denote its $G$-invariant part.
	
	For a homomorphism $L \to G$ of linear algebraic groups, we write $\Ind^G_L$ for the right adjoint functor to the forgetful functor from the category of representations of $G$ to that of $L$.
	
	For a connected reductive algebraic group $G$ over an algebraically closed field of characteristic zero and a maximal torus $H\subset G$, we regard coroots of $(G,H)$ as elements of $\fh$ by differential.

	\subsection{Linear algebra and matrices}
	For a vector space $V$, let $V^\vee$ denote its dual. We denote the canonical pairing of $V$ with $V^\vee$ by $\langle-,-\rangle$.
	
	For a vector space $V$ over a field $F$, let $\dim_F V$ denote the dimension of $V$ over $F$.
	
	For a field $F$ and a finite set $S$, let $M_S(F)$ denote the $F$-algebra of square matrices indexed by elements of $S$ with entries in $F$.

	A block diagonal matrix with blocks $(A_1, A_2, \dots, A_n)$ from the top-left to the bottom-right will be denoted by $\diag(A_1, A_2, \dots, A_n)$.

	\section{Loewy's classification scheme and Borel--Tits' criterion}\label{sec:lbt}
	
	Let $F'/F$ be a (possibly infinite) Galois extension of fields. In this section, we present an axiomatic approach to the work of \cite{MR1500635} and \cite[Section~12]{MR207712}.

	\subsection{$\Gamma$-categories and descent systems}\label{sec:gamma_cat}

	Throughout this section, let $\Gamma$ be a group. Here, we review generalities on $\Gamma$-categories and systems of Galois descent following Grothendieck.

	\begin{defn}
		\begin{enumerate}
			\item A $\Gamma$-category consists of a category $\mathcal{C}$, an endofunctor ${}^\sigma(-)$ of $\mathcal{C}$ for each $\sigma \in \Gamma$, an isomorphism $u : {}^e(-) \cong \mathrm{id}_{\mathcal{C}}$, and a natural isomorphism $\mu_{\sigma,\tau} : {}^\sigma(-) \circ {}^\tau(-) \cong {}^{\sigma\tau}(-)$ for each pair $(\sigma,\tau) \in \Gamma^2$, satisfying the following properties:
			\begin{enumerate}[label=(GC\arabic*)]
				\item\label{con:ass} $\mu_{\sigma,\tau\rho} \cdot (\id_{{}^\sigma(-)} \circ \mu_{\tau,\rho}) = \mu_{\sigma\tau,\rho} \cdot (\mu_{\sigma,\tau} \circ \id_{{}^\rho(-)})$
				for $\rho, \sigma, \tau \in \Gamma$,
				\item\label{con:unit} $\mu_{e,\sigma} = u \circ \id_{{}^\sigma(-)}$ and $\mu_{\sigma,e} = \id_{{}^\sigma(-)} \circ u$ for $\sigma \in \Gamma$.
			\end{enumerate}
			We will sometimes say that $\cC$ is a $\Gamma$-category for short.
			\item An object $X \in \cC$ of a $\Gamma$-category $(\cC, {}^\sigma(-), u, \mu_{\sigma,\tau})$ is called \emph{self-conjugate} if there exist isomorphisms ${}^\sigma X \cong X$ for all $\sigma \in \Gamma$.
			\item We say that a $\Gamma$-category $(\mathcal{C}, {}^\sigma(-), u, \mu_{\sigma,\tau})$ is \emph{$F$-linear} (resp.\ \emph{$F'/F$-linear}) if $\mathcal{C}$ is equipped with the structure of an $F$-linear (resp.\ $F'$-linear) category and the endofunctors ${}^\sigma(-)$ are $F$-linear.
			\item An $F'/F$-linear $\Gamma_{F'/F}$-category $(\mathcal{C}, {}^\sigma(-), u, \mu_{\sigma,\tau})$ is called \emph{semilinear} if
			\[
			{}^\sigma (cf) = \sigma(c) \, {}^\sigma f
			\]
			for $\sigma \in \Gamma$, $c \in F'$, and a morphism $f$ of $\mathcal{C}$.
			\item We say that a $\Gamma$-category is \emph{locally small} (resp.\ \emph{essentially small}) if its underlying category is locally small (resp.\ essentially small).
			\item We say that a $\Gamma$-category $(\mathcal{C}, {}^\sigma(-), u, \mu_{\sigma,\tau})$ is \emph{abelian} if $\mathcal{C}$ is abelian and ${}^\sigma(-)$ are additive.
			\item We say that an abelian $\Gamma$-category is \emph{Noetherian} (resp.\ \emph{Artinian}) if its underlying abelian category is Noetherian (resp.\ Artinian).
		\end{enumerate}
	\end{defn}
	
	\begin{rem}
		The definitions of $\Gamma$-categories depend on the literature for the coherent structures (e.g.\ \cite{MR0711065,MR0861356,MR1126178}). We adopt the most relaxed version here.
	\end{rem}
	
		\begin{ex}
		Let $(\fg,K)$ be a (Harish-Chandra) pair over $F$ in the sense of \cite{MR3853058}. Then the category of $(F' \otimes_F \fg, F' \otimes_F K)$-modules is a locally small semilinear abelian $\Gamma_{F'/F}$-category for base changes by $\sigma \in \Gamma$ (see \cite[Section~3.1]{MR3853058}).
		
	\end{ex}
	
	Let us record an elementary result that we will use later:

	\begin{lem}\label{lem:self-conjugacy}
		Let $\cC$ be a $\Gamma$-category. Then the self-conjugacy property is stable under the formation of isomorphisms.
		
	\end{lem}
	
	\begin{proof}
		Let $M, N \in \cC$ be objects with an isomorphism $f: M \cong N$. Suppose that $M$ is self-conjugate. For each $\sigma \in \Gamma$, fix an isomorphism $\varphi_\sigma: {}^\sigma M \cong M$. Then $f \circ \varphi_\sigma \circ {}^\sigma f^{-1}$ gives an isomorphism ${}^\sigma N \cong N$. This shows that $N$ is self-conjugate.
		
	\end{proof}
	
	\begin{defn}\label{defn:descent_system}

		Let $(\cC, {}^\sigma(-), u, \mu_{\sigma,\tau})$ be a $\Gamma$-category.
		\begin{enumerate}
			\item A pair consisting of an object $M \in \cC$ and a family of isomorphisms 
			$\varphi_\sigma: {}^\sigma M \cong M$ indexed by $\sigma \in \Gamma$ is called a descent system if
			\begin{enumerate}[label=(DS)]
				\item \label{con:DS} the equalities $\varphi_{\sigma\tau} \circ \mu_{\sigma,\tau,M} = \varphi_\sigma \circ {}^\sigma \varphi_\tau$ hold for $\sigma, \tau \in \Gamma$.
			\end{enumerate}
			\item A morphism $(M, \varphi_\sigma) \to (N, \psi_\sigma)$ of descent systems is a map $f: M \to N$ in $\cC$ such that $f \circ \varphi_\sigma = \psi_\sigma \circ {}^\sigma f$ for $\sigma \in \Gamma$.
			\item We denote the category of descent systems of $\cC$ by $\DS(\cC)$. We remark that if $\cC$ is $F$-linear, then $\DS(\cC)$ is naturally endowed with the structure of an $F$-linear category.
		\end{enumerate}
		
	\end{defn}
	
	\begin{ex}

		Let $(\cC, {}^\sigma(-), u, \mu_{\sigma,\tau})$ be a $\Gamma$-category, and let $\Gamma'$ be a subgroup of $\Gamma$. Then $\cC$ is naturally endowed with the structure of a $\Gamma'$-category by restriction. If $\Gamma = \Gamma_{F'/F}$ and $\Gamma' = \Gamma_{F'/F''}$ for a finite extension $F''/F$ in $F'$, then we denote the category of descent systems of $\cC$ as a $\Gamma_{F'/F''}$-category by $\DS(\cC; F'/F'')$.
		
	\end{ex}

	\begin{ex}\label{ex:Gamma={e}}
		Suppose that $\Gamma$ is trivial. Then $\pi:\DS(\cC)\to\cC$ is an isomorphism. Its inverse is given by $X\mapsto (X,u_X)$.
	\end{ex}
	
	\begin{prop}\label{prop:(co)limit}
		Let $\cC$ be a $\Gamma$-category. 
		\begin{enumerate}
			\item The forgetful functor $\pi:\DS(\cC)\to \cC$ is conservative, i.e., $\pi$ reflects isomorphisms.
			\item (Co)limits in $\DS(\cC)$ are computed in $\cC$ if they exist in $\cC$.
			\item If $\cC$ is abelian, so is $\DS(\cC)$.
		\end{enumerate}
	\end{prop}
		
	\begin{proof}
		Part (1) is elementary.
		
		For (2), let $\mathscr{I}$ be a category, and $F:\mathscr{I}\to\DS(\cC)$ be a functor. For $i\in\mathscr{I}$, write $F(i)=(\pi(F(i)),\psi^i_\sigma)$. Suppose that $F$ admits a limit in $\cC$, that is, there is a terminal object $(N,q_i)$ in the category of cones for $\pi\circ F$ (\cite[Chapter III, Section~3]{MR1712872}). Then it is easy to check that $({}^\sigma N,\psi^i_{\sigma}\circ {}^\sigma q_i)$ is a terminal cone for $\sigma\in\Gamma$. We thus obtain an isomorphism $\psi_\sigma:{}^\sigma N\cong N$ satisfying $q_i\circ\psi_\sigma=\psi^i_{\sigma}\circ {}^\sigma q_i$ for each $i\in \mathscr{I}$. It is easy to show that $(N,\psi_\sigma)$ exhibits a limit descent system of $F$. The assertion for colimits is verified in the dual way.
		
		Part (3) is an easy consequence of (1) and (2).
	\end{proof}

	The notion of descent systems is fine if $\Gamma$ is finite:
	
	\begin{ex}
		Let $G$ be an affine group scheme over a field $F$, and $F'$ be a finite Galois extension of $F$. Then the base change $F'\otimes_F(-)$ determines an equivalence
		\[\Rep G\simeq \DS(\Rep(F'\otimes_FG).\]
		A similar equivalence holds true for the categories of possibly infinite-dimensional representations. We have similar equivalences for a flat affine group scheme $G$ over a commutative ring $k$ and a Galois extension $k'/k$ of commutative rings in the sense of \cite[Definition 1.3.1]{MR4627704}.
	\end{ex}
	
	If $\Gamma$ is infinite, we have to be careful with the effectivity (continuity) in applications. 
	
	\begin{ex}
		Consider the category of $F'$-vector spaces. It is a $\Gamma_{F'/F}$-category for the usual Galois twists. The effectivity of a descent system $(V',\varphi_\sigma)$ means that $(V',\varphi_\sigma)$ is isomorphic to the descent system obtained by the base change of an $F$-vector space. This is equivalent to asking whether the semilinear action of $\Gamma_{F'/F}$ on $V'$ corresponding to $(\varphi_\sigma)$ is continuous. However, this does not hold true in general if $\Gamma_{F'/F}$ is infinite.
	\end{ex}
	
	We find another effectivity issue from the following result:
	
	\begin{prop}\label{prop:res}
		Let $\cC$ be a $\Gamma$-category. Suppose that $\cC$ admits products indexed by $\Gamma$. Then the forgetful functor $\DS(\cC)\to\cC$ admits a right adjoint functor. Moreover, if $\cC$ is an $F$-linear $\Gamma$-category, then the adjunction is $F$-linear.
	\end{prop}
	
	If $\Gamma=\Gamma_{F'/F}$ then we will denote this right adjoint functor by $\Res_{F'/F}$.
	
	\begin{proof}
		For $N\in\cC$, we define a descent system $(\psi^0_\sigma)$ on $\prod_{\sigma\in\Gamma} {}^\sigma N$ by switching the factors. In fact, ${}^\sigma(-)$ commutes with products since ${}^\sigma(-)$ is an auto-functor. Then use $\mu_{\sigma,\tau,N}$ to obtain an isomorphism
		${}^\sigma\prod_{\tau\in\Gamma}{}^\tau N
		\cong \prod_{\tau\in\Gamma} {}^{\sigma\tau}N$.
		
		Let $(M,\varphi_\sigma)\in\DS(\cC)$. Regard 
		$\Hom_{\DS(\cC)}\left((M,\varphi_\sigma),\left(\prod_{\sigma\in\Gamma}{}^\sigma N,\varphi^0_\sigma\right)\right)$ as a subset of $\prod_{\sigma\in\Gamma}\Hom_\cC \left(M,\prod_{\sigma\in\Gamma}{}^\sigma N\right)$ by the forgetful map
		\[\Hom_{\DS(\cC)}\left((M,\varphi_\sigma),\left(\prod_{\sigma\in\Gamma}{}^\sigma N,\varphi^0_\sigma\right)\right)
		\hookrightarrow \Hom_{\cC}\left(M,\prod_{\sigma\in\Gamma}{}^\sigma N\right)\]
		and the universality of the product.
		Then we define a bijection
		\[\Hom_{\DS(\cC)}\left((M,\varphi_\sigma),
		\left(\prod_{\sigma\in\Gamma}{}^\sigma N,\varphi^0_\sigma\right)\right)
		\cong\Hom_{\cC}(M,N)\]
		by $(f_{\sigma})\mapsto f_e$ and $f\mapsto (f_\sigma)$, where $f_\sigma$ is given by
		the composition
		$M\overset{\varphi^{-1}_\sigma}{\cong}
		{}^\sigma M\overset{{}^\sigma f}{\to} {}^\sigma N$.
		This gives the structure of the adjunction.
		
		The last statement follows by construction.
	\end{proof}
	
	Observe that the canonical map 
	$F'\otimes_F F'\to \prod_{\sigma\in\Gamma_{F'/F}} F';~a\otimes b\mapsto (\sigma(a)b)$
	is not an isomorphism unless $\Gamma$ is finite. This is evidence why descent systems of $\cC$ should not be effective when $\Gamma$ is infinite and $\cC$ admits infinite products.
	
	\subsection{Loewy pairs and examples}\label{sec:Loewy_pair}
	
	In this section, we introduce some notion which enables us to perform Loewy's arguments possibly in the setting of infinite Galois extensions.
	
	Informally speaking, if one wishes to study a certain $F$-linear category $\cC_F$ through ``Galois descent'' from a $\Gamma_{F'/F}$-category $\cC$, we need a full subcategory of $\DS(\cC)$ which is expected to be equivalent to $\cC_F$ because of the effectivity issue discussed at the end of the previous section.
		
	\begin{defn}\label{def:lp}
		\begin{enumerate}
			\item A Loewy pair consists of
			\begin{itemize}
				\item a locally small and essentially small $F'/F$-linear abelian $\Gamma_{F'/F}$-category $\cC$ and
				\item an abelian full subcategory $\DD(\cC)$ of $\DS(\cC)$
			\end{itemize}
			such that
			\begin{enumerate}[label=(LP\arabic*)]
				\item \label{con:fin_length} every object of $\cC$ is of finite length,
				\item \label{con:subobject_inv} $\DD(\cC)$ is closed under formation of subobjects in $\DS(\cC)$,
				\item \label{con:bcthm} for $(M,\varphi_\sigma),(N,\psi_\sigma)\in\DD(\cC)$, the canonical map
				\[F'\otimes_F \Hom_{\DD(\cC)}((M,\varphi_\sigma),(N,\psi_\sigma))
				\to \Hom_{\cC}(M,N)\]
				is an isomorphism, and
				\item \label{con:surj} for every simple object $Q\in\cC$, there exists a simple object $(M,\varphi_\sigma)$ in $\DD(\cC)$ such that $\Hom_{\cC}(M,Q)\neq 0$.
			\end{enumerate}
			\item An $F'/F$-linear abelian $\Gamma_{F'/F}$-category $\cC$ is called \emph{Loewy} if $\Gamma_{F'/F}$ is finite and $(\cC,\DS(\cC))$ is a Loewy pair.
		\end{enumerate}
	\end{defn}
	
	\begin{conv}
		If we are given a group $\Gamma$, an abelian $\Gamma$-category $\cC$, and an abelian full subcategory $\DD(\cC)$ of $\DS(\cC)$, objects of $\DD(\cC)$ will be called descent data.
	\end{conv}

	For the definitions, let us note:

	\begin{prop}\label{prop:gamma_action}
		Let $\cC$ be a locally small $\Gamma_{F'/F}$-category. Consider descent systems $(M,\varphi_\sigma)$ and $(N,\psi_\sigma)$.
		\begin{enumerate}
			\item The group $\Gamma_{F'/F}$ acts on $\Hom_{\cC}(M,N)$ by
			$\sigma\cdot f \coloneqq\psi_\sigma\circ {}^\sigma f\circ \varphi^{-1}_\sigma$ for $\sigma\in\Gamma_{F'/F}$ and $f\in \Hom_{\cC}(M,N)$.
			\item A morphism $f:M\to N$ respects the descent data if and only if it is $\Gamma_{F'/F}$-invariant for the action in (1).
			\item If $\cC$ is $F$-linear, so is the action in (1).
			\item Assume that $\cC$ is semilinear. For descent systems $(M,\varphi_\sigma)$ and $(N,\psi_\sigma)$,
			the canonical map
			\begin{equation}
				F'\otimes_F \Hom_{\DS(\cC)}((M,\varphi_\sigma),(N,\psi_\sigma))
				\to \Hom_{\cC}(M,N)
				\label{eq:bcmap}	
			\end{equation}
			is $\Gamma_{F'/F}$-equivariant, where $\Gamma_{F'/F}$ acts on the domain by $\sigma\cdot (c\otimes f)=\sigma(c)\otimes f$.
		\end{enumerate}	
	\end{prop}
	
	\begin{proof}
		Part (1) is straightforward:
		\[\begin{split}
			\psi_\sigma\circ {}^\sigma (\psi_\tau\circ {}^\tau f\circ \varphi^{-1}_\tau)\circ \varphi^{-1}_\sigma
			&=\psi_\sigma\circ {}^\sigma\psi_{\tau}\circ {}^\sigma({}^\tau f)\circ 
			{}^\sigma \varphi^{-1}_\tau\circ \varphi^{-1}_\sigma\\
			&\overset{\mathrm{\ref{con:ass}}}{=}\psi_{\sigma\tau}\circ\mu_{\sigma,\tau,M}
			\circ {}^\sigma({}^\tau f)\circ \mu^{-1}_{\sigma,\tau,M}\circ \varphi^{-1}_{\sigma\tau}\\
			&\overset{\mathrm{\ref{con:DS}}}{=}\psi_{\sigma\tau}
			\circ {}^{\sigma\tau} f\circ \varphi^{-1}_{\sigma\tau}
		\end{split}\]
		for $f\in \Hom_{\cC}(M,N)$ and
		$\sigma,\tau\in\Gamma_{F'/F}$;
		\[\begin{split}
			\varphi_e&=u_{M} \circ {}^e\varphi_e \circ u^{-1}_{{}^e M}\\
			&\overset{\mathrm{\ref{con:unit}}}{=}u_{M} \circ {}^e\varphi_e \circ \mu^{-1}_{e,e,M}\\
			&=u_{M} \circ \varphi_e^{-1}\circ (\varphi_e\circ {}^e\varphi_e) \circ \mu^{-1}_{e,e,M}\\
			&\overset{\mathrm{\ref{con:DS}}}{=}u_M,
		\end{split}\]
		\[\psi_e\circ{}^ef\circ \varphi^{-1}_e
		=\psi_e \circ u^{-1}_N\circ f\circ u_M\circ \varphi^{-1}_e =f\]
		for $f\in \Hom_{\cC}(M,N)$.
		The rest follows by unwinding the definitions.
	\end{proof}
	
	\begin{cor}\label{cor:bc_criterion}
		Let $\cC$ be a locally small semilinear $\Gamma_{F'/F}$-category. Then the map \eqref{eq:bcmap} is an isomorphism if and only if the action of $\Gamma_{F'/F}$ on $\Hom_{\cC}(M,N)$ is continuous. In particular, \eqref{eq:bcmap} is an isomorphism if $\Gamma_{F'/F}$ is finite.
	\end{cor}

	\begin{prop}\label{prop:nosimple}
		A pair of an abelian $\Gamma_{F'/F}$-category $\cC$ with Condition \ref{con:fin_length} and 
		an abelian full subcategory $\DD(\cC)\subset\DS(\cC)$ satisfies \ref{con:surj} if and only if for every simple object $Q\in\cC$, there exists a descent datum $(M,\varphi_\sigma)\in\DD(\cC)$ such that $\Hom_{\cC}(M,Q)\neq 0$.
	\end{prop}
	
	\begin{proof}
		The ``only if'' direction is clear. Suppose that we are given a simple object $Q\in\cC$, $(M,\varphi_\sigma)\in\DD(\cC)$, and a nonzero morphism $p:M\to Q$. We construct an increasing sequence
		\begin{equation}
			(M_0,\varphi_\sigma)\hookrightarrow (M_2,\varphi_\sigma)\hookrightarrow\cdots
			\label{eq:increasing_seq}
		\end{equation}
		of subobjects of $(M,\varphi_\sigma)$ in $\DD(\cC)$ with the following properties:
		\begin{enumerate}
			\item[(i)] $M_0=0$,
			\item[(ii)] the sequence \eqref{eq:increasing_seq} stops if and only if $p|_{M_i}\neq 0$,
			\item[(iii)] the successive quotients of \eqref{eq:increasing_seq} are simple or 0.
		\end{enumerate}
		
		We set $M_0=0$. This determines a subobject of $(M,\varphi_\sigma)$. Let $i\geq 1$. We set $M_i=M_{i-1}$ if $p|_{M_{i-1}}\neq 0$. Suppose that $p|_{M_{i-1}}=0$. Since $p$ is nonzero, so is $M/M_{i-1}$. Since $M/M_{i-1}$ is of finite length, one can find a simple subobject $(\bar{M}_i,\bar{\varphi}_\sigma)\in\DD(\cC)$ of the quotient datum $(M_i/M_{i-1},\bar{\varphi}_i)\coloneqq (M_i,\varphi_i)/(M_{i-1},\varphi_\sigma)$.
		Let $(M_i,\varphi_\sigma)$ be its preimage in $(M,\varphi_\sigma)$. It is evident by construction that the resulting sequence satisfies the properties (i) - (iii) mentioned above.
		
		The sequence \eqref{eq:increasing_seq} does stop since $M$ is of finite length. In view of (i) and (ii), one can find a (unique) positive integer $n\geq 1$ such that $p|_{M_{n-1}}=0$ and $p|_{M_n}\neq 0$. This implies $\Hom_{\cC}(M_n/M_{n-1},Q)\neq 0$. In particular, $M_n/M_{n-1}$ is nonzero. Property (iii) now implies that $(M_n/M_{n-1},\bar{\varphi}_{\sigma})$ is simple. This completes the proof.
	\end{proof}

	\begin{cor}\label{cor:finite_loewy}
		Suppose that $\Gamma_{F'/F}$ is finite. A locally small essentially small semilinear abelian $\Gamma_{F'/F}$-category is Loewy if and only if every object of $\cC$ is of finite length.
	\end{cor}
	
	\begin{proof}
		The problem is to verify Condition \ref{con:surj}. Recall that the underlying object of $\Res_{F'/F} Q$ is $\prod_{\sigma\in\Gamma_{F'/F}} {}^\sigma Q$ for an object $Q\in \cC$. In particular, the underlying object of $\Res_{F'/F} Q$ is nonzero if and only if so is $Q$. The assertion now follows since the counit satisfies the equivalent conditions in Proposition \ref{prop:nosimple}. 
	\end{proof}
	
	\begin{rem}
		We see from the construction of $\Res_{F'/F} Q$ that its underlying object is of finite length. Hence the argument of Corollary \ref{cor:finite_loewy} works without the finite length condition. However, we still use the Artin condition later (Lemma \ref{lem:well-defined}). Hence for our results on Loewy $\Gamma_{F'/F}$-categories in this paper, we could replace \ref{con:fin_length} with `$\cC$ Artinian'. Similarly, we could slightly weaken the definition of Loewy pairs. We leave them to the readers who want a weaker definition.
	\end{rem}
	
	\begin{rem}
		For a finite Galois extension $k'/k$ of commutative rings of Galois group $\Gamma_{k'/k}$, one could define the notion of $k$-linear (resp.~$k'/k$-linear, semilinear) $\Gamma_{k'/k}$-categories along the same line. Then Propositions \ref{prop:res} and \ref{prop:gamma_action} - Corollary \ref{cor:finite_loewy} hold true in this setting.
	\end{rem}

	\begin{ex}[\cite{MR1500635}]\label{ex:Loewy}
		Let $G$ be a group. Then the category of finite-dimensional complex representations of $G$ with the complex conjugate action is Loewy. The category of descent systems is equivalent to that of finite-dimensional real representations of $G$.
	\end{ex}
	
	\begin{ex}\label{ex:hc_fin}
		Put $F'/F=\bC/\bR$. Let $(\fg,K)$ be a reductive pair in the sense of \cite[Definition 4.30]{MR1330919}. Recall that a $(\fg,K)$-module is called finitely generated (resp.~$Z(\fg)$-finite) if it is so as a $U(\fg)$-module (resp.~$\Ann_{Z(\fg)}(V)$ is of finite codimension in $Z(\fg)$). Let $(\fg,K)\cmod$ be the category of $(\fg,K)$-modules. We denote its full subcategory consisting of finitely generated and $Z(\fg)$-finite $(\fg,K)$-modules by
		$(\fg,K)\cmod_{\fingen,\Zf}$.
		It is evident by definition that $(\fg,K)\cmod_{\fingen,\Zf}$ is closed under formations of subquotient in $(\fg,K)\cmod$, in particular, it is an abelian subcategory of $(\fg,K)\cmod$. Moreover, $(\fg,K)\cmod_{\fingen,\Zf}$ is a Loewy $\Gamma_{\bC/\bR}$-category for the complex conjugation since every finitely generated and $Z(\fg)$-finite $(\fg,K)$-module is of finite length by Harish-Chandra's finiteness theorem (see \cite[Corollary 7.223]{MR1330919} and \cite[3.4.1. Theorem]{MR929683}). Let $\fg_0$ denote the given real form of $\fg$. Then we can naturally define the notion of finitely generated and $Z(\fg_0)$-finite $(\fg_0,K)$-modules. Moreover, the category of descent systems of $(\fg,K)\cmod_{\fingen,\Zf}$ is equivalent to that of finitely generated and $Z(\fg_0)$-finite $(\fg_0,K)$-modules. One could define the notion of $U(\fg)^K$-finiteness in the obvious way as a variant, but we remark that for a finitely generated $(\fg,K)$-module, the following conditions are equivalent by \cite[3.4.1. Theorem]{MR929683} and a similar argument to \cite[Chapter VII, Section~2, EXAMPLES 2)]{MR1330919}:
		\begin{enumerate}
			\item[(a)] $V$ is $Z(\fg)$-finite;
			\item[(b)] $V$ is $Z(\fg)^K$-finite;
			\item[(c)] $V$ is admissible.
		\end{enumerate}
		We find as a consequence that every irreducible $(\fg,K)$-module $V$ is finitely generated and $Z(\fg)$-finite since $Z(\fg)^K$ acts on $V$ as a character (cf.~\cite[3.3.2. Lemma]{MR929683}).
	\end{ex}
	
	\begin{cor}\label{cor:fl}
		Suppose that $\Gamma_{F'/F}$ is finite. Let $\cC$ be a locally small semilinear abelian $\Gamma_{F'/F}$-category. Then its full subcategory $\cC_{\mathrm{fl}}$ consisting of objects of finite length is Loewy for the action of $\Gamma_{F'/F}$ on $\cC$ if $\cC_{\mathrm{fl}}$ is essentially small. Moreover, we have $\DS(\cC_{\mathrm{fl}})=\DS(\cC)_{\mathrm{fl}}$.
	\end{cor}
	
	\begin{proof}
		Since the functors ${}^\sigma(-)$ and ${}^{\sigma^{-1}}(-)$ are quasi-inverse to each other, ${}^\sigma(-)$ respect objects of finite length. This shows the first part. 
		
		Since $\pi:\DS(\cC)\to\cC$ is exact and conservative, $\pi$ reflects the finite length property. This shows $\DS(\cC_{\mathrm{fl}})\subset \DS(\cC)_{\mathrm{fl}}$. To prove the converse containment, consider a descent system $(M,\varphi_\sigma)$ of finite length. We define an isomorphism
		$\Res_{F'/F} M\cong \oplus_{\tau\in\Gamma_{F'/F}} (M,\varphi_\sigma)$
		summand-wisely by $(\varphi_\tau)$. In particular, $\Res_{F'/F} M$ is of finite length. Since $\Res_{F'/F}$ is exact and conservative by construction, we conclude that $M$ is of finite length. This completes the proof.
	\end{proof}

	It remains difficult to check Condition \ref{con:surj} directly in examples if $\Gamma_{F'/F}$ is infinite. Informally speaking, this should be reduced to Corollary \ref{cor:finite_loewy} by showing that any (simple) object $M$ of $\cC$ is defined over $F''\in\Lambda_{F'/F}$.
	
	Recall that we have the restriction map $(-)|_{F''}:\Gamma_{F'/F}\twoheadrightarrow \Gamma_{F''/F}$.
	
	\begin{cons}\label{cons:basechange}
		Let $F''\in\Lambda_{F'/F}$, $\cC_{F''}$ be a $\Gamma_{F''/F}$-category, and \[F'\otimes_{F''}(-):\cC_{F''}\to\cC\]
		be a functor. Suppose that we are given an isomorphism
		\[\alpha_{\sigma}:{}^\sigma(F'\otimes_{F''}(-))\cong F'\otimes_{F''} {}^{\sigma|_{F''}}(-).\]
		
		Let $M\in\cC_{F''}$ and $\varphi_{\sigma|_{F''}}:{}^{\sigma|_{F''}} M\cong M$. Then we set
		\[\varphi_{\sigma}\coloneqq (F'\otimes_{F''}\varphi_{\sigma|_{F''}})\circ \alpha_{\sigma,M}:{}^\sigma (F'\otimes_{F''} M)\cong F'\otimes_{F''} M.\]
	\end{cons}
	
	\begin{lem}\label{lem:baseupds}
		Consider the setting of Construction \ref{cons:basechange}. Suppose that $\alpha_\sigma$ is associative in $\sigma$, i.e., the following diagram commutes for any $\sigma,\tau\in\Gamma_{F'/F}$:
		\begin{equation}
			\begin{tikzcd}
				&F'\otimes_{F''}{}^{(\sigma\tau)|_{F''}}(-)&\\
				F'\otimes_{F''}{}^{\sigma|_{F''}}({}^{\tau|_{F''}}(-))
				\ar[ru, "F'\otimes_{F''}\mu_{\sigma|_{F''},\tau|_{F''}}"]
				&&{}^{\sigma\tau}(F'\otimes_{F''}(-))
				\ar[lu, "\alpha_{\sigma\tau}"']\\
				{}^\sigma(F'\otimes_{F''} {}^{\tau|_{F''}}(-))
				\ar[u, "\alpha_{\sigma,{}^{\tau|_{F''}} (-)}"]	
				&&{}^\sigma({}^\tau(F'\otimes_{F''}(-)))
				\ar[ll, "{}^\sigma\alpha_{\tau}"']\ar[u, "\mu_{\sigma,\tau}"'].
			\end{tikzcd}\label{diag:associative}
		\end{equation}
		Then Construction \ref{cons:basechange} determines a functor $F'\otimes_{F''}(-):\DS(\cC_{F''})\to \DS(\cC)$.
	\end{lem}
	
	\begin{proof}
		For a descent system on $M\in\cC_{F''}$, it is easy to show that the maps naturally form a descent system on $F'\otimes_{F''}M$.
	\end{proof}
	
	\begin{defn}\label{defn:Ldatum}
		A Loewy datum consists of
		\begin{itemize}
			\item a locally small and essentially small semilinear abelian $\Gamma_{F'/F}$-category $\cC$,
			\item an abelian full subcategory $\DD(\cC)\subset\DS(\cC)$,
			\item a semilinear locally small and essentially small abelian $\Gamma_{F''/F}$-category $\cC_{F''}$
			for each $F''\in\Lambda_{F'/F}$,
			\item an $F''$-linear functor $F'\otimes_{F''}(-):\cC_{F''}\to \cC$
			for $F''\in\Lambda_{F'/F}$,
			\item a natural transformation $\alpha^{F''}_{\sigma}:{}^\sigma(F'\otimes_{F''}(-))\cong F'\otimes_{F''} {}^{\sigma|_{F''}}(-)$ for each field $F''\in\Lambda_{F'/F}$ and $\sigma\in\Gamma_{F'/F}$,
		\end{itemize}
		such that
		\begin{enumerate}[label=(LD\arabic*)]
			\item \label{con:fin_length_Ld} every object of $\cC$ is of finite length,
			\item \label{con:subobject_inv_Ld} $\DD(\cC)$ is closed under formation of subobjects in $\DS(\cC)$,
			\item \label{con:bcthm_Ld} for $F''\in\Lambda_{F'/F}$ and $M,N\in\cC_{F''}$, the canonical map
			\begin{equation}
				F'\otimes_{F''} \Hom_{\cC_{F''}}(M,N)
				\to \Hom_{\cC}(F'\otimes_{F''}M,F'\otimes_{F''}N)\label{eq:bcmap_Ld}
			\end{equation}
			is an isomorphism,
			\item \label{con:associative_Ld} $\alpha^{F''}_\sigma$ is associative in $\sigma$ in the sense of Lemma \ref{lem:baseupds} for each $F''\in\Lambda_{F'/F}$,
			\item \label{con:continuity_Ld} the functor $\DS(\cC_{F''})\to \DS(\cC)$ of Lemma \ref{lem:baseupds} is essentially surjective onto $\DD(\cC)$, and that
			\item \label{con:ess_surj_Ld} for every simple object $Q\in\cC$, there exist $F''\in\Lambda_{F'/F}$ and $Q_{F''}\in\cC_{F''}$ such that
			$F'\otimes_{F''} Q_{F''}\cong Q$.
		\end{enumerate}
		If there is no risk of confusion, we say that a Loewy datum on a semilinear abelian $\Gamma_{F'/F}$-category $\cC$ is given. In this case, we follow the above for the symbols.
	\end{defn}
	
	\begin{ex}[{\cite[Section~12]{MR207712}}]\label{ex:repG}
		Let $G$ be an affine group scheme over $F$. Define $\DD(\Rep(F'\otimes_FG))\subset \DS(\Rep(F'\otimes_FG))$ as the full subcategory consisting of descent systems with the property that the corresponding semilinear actions of $\Gamma_{F'/F}$ are continuous.
		Then $\Rep(F''\otimes_F G)$ with $F''\in\Lambda_{F'/F}\cup\{F'\}$ and $\DD(\Rep(F'\otimes_FG))$ form a Loewy datum in the standard manner.
	\end{ex}
	
	According to Corollary \ref{cor:fl}, preservation of the finite length property by base change along finite Galois extensions is expected. However, to verify that for infinite Galois extensions may be hard. This will be an obstacle for finding more examples of Loewy data.

	\begin{thm}\label{thm:Ld->Lp}
		Let $(\cC,\DD(\cC),\cC_{F''},F'\otimes_{F''}(-),\alpha^{F''}_\sigma)$ be a Loewy datum.
		\begin{enumerate}
			\item For $F''\in\Lambda_{F'/F}$, the functor $F'\otimes_{F''}(-):\DS(\cC_{F''})\to \DD(\cC)$ is fully faithful.
			\item The pair $(\cC,\DD(\cC))$ is Loewy.
		\end{enumerate}
	\end{thm}
	
	In the below, we write a given element of $\Gamma_{F''/F}$ as $\bar{\sigma}$.

	\begin{proof}
		Let $(M,\varphi_{\bar{\sigma}}),(N,\psi_{\bar{\sigma}})\in\DS(\cC_{F''})$. Let
		\[\begin{array}{cc}
			(F'\otimes_{F''}M,\varphi_\sigma),&
			(F'\otimes_{F''}N,\psi_\sigma)
		\end{array}\]
		denote the descent data in $\cC$ obtained by Construction \ref{cons:basechange}.
		We put actions of $\Gamma_{F''/F}$ and $\Gamma_{F'/F}$ on $\Hom_{\cC_{F''}}(M,N)$ and 
		$\Hom_{\cC}(F'\otimes_{F''}M,F'\otimes_{F''}N)$ respectively by Proposition \ref{prop:gamma_action} (1). Then the map \eqref{eq:bcmap_Ld} is $\Gamma_{F'/F}$-equivariant.
		Take the $\Gamma_{F'/F}$-invariant parts to obtain an $F$-linear isomorphism
		\begin{equation}
			\Hom_{\DS(\cC_{F''})}((M,\varphi_{\bar{\sigma}}),(N,\psi_{\bar{\sigma}}))
			\cong \Hom_{\DD(\cC)}((F'\otimes_{F''}M,\varphi_\sigma),(F'\otimes_{F''}N,\psi_\sigma)).
			\label{eq:DD_ff}
		\end{equation}
		This proves (1).
		
		We next prove (2). Condition \ref{con:bcthm} follows from \ref{con:ess_surj_Ld}, \ref{con:bcthm_Ld}, and \eqref{eq:DD_ff} (put $F''=F$). It remains to verify \ref{con:surj}. Let $Q$ be a simple object of $\cC$. Choose a field $F''\in\Lambda_{F'/F}$, $Q'\in \cC_{F''}$, and an isomorphism $F'\otimes_{F''}Q'\cong Q$. In particular, $Q$ is nonzero. We now deduce
		\[\begin{split}
			\Hom_\cC(F'\otimes_{F''}(F''\otimes_F\Res_{F''/F}Q'),Q)
			&\cong F'\otimes_{F''}\Hom_{\cC_{F''}}(F''\otimes_F \Res_{F''/F} Q',Q')\\
			&\cong F'\otimes_{F''}\End_{\DS(\cC_{F''})}(\Res_{F''/F}Q')\\
			&\neq 0
		\end{split}\]
		(recall the construction of $\Res_{F''/F}$ for the nonvanishing of $\Res_{F''/F}Q'$).
	\end{proof}
	
	\begin{rem}\label{rem:overcondition}
		For the proof of (1), we only used \ref{con:bcthm_Ld} and \ref{con:associative_Ld}.			
	\end{rem}
	
	\begin{rem}
		We notice that Conditions \ref{con:fin_length_Ld} and \ref{con:subobject_inv_Ld} are not used in the proof of Theorem \ref{thm:Ld->Lp} except checking \ref{con:fin_length} and \ref{con:subobject_inv} respectively.
	\end{rem}

	The aim of Section~\ref{sec:lbt} is to construct a bijection \[\Simple(\DD(\cC))\cong\Gamma_{F'/F}\backslash\Simple(\cC).\]
	We also discuss existence of descent data on simple objects of $\cC$. Finally, we study the division algebras of endomorphisms of simple descent data, based on these two results.
	
	\subsection{Classification of simple objects}\label{sec:Loewy}
	
	Consider a Loewy pair $(\cC,\DD(\cC))$. The goal of this section is to prove:
	
	\begin{thm}\label{thm:Loewy}
		There is a unique bijection
		$\Simple(\DD(\cC))\cong\Gamma_{F'/F}\backslash \Simple(\cC)$ with the following property:
		\begin{enumerate}[label=(Lb)]
			\item \label{con:Loewy_bij} simple objects $(M,\varphi_\sigma)\in \DD(\cC)$ and $M'\in\cC$ are in correspondence if and only if $\Hom_{\cC}(M,M')$ is nonzero.
		\end{enumerate}
	\end{thm}

	\begin{lem}\label{lem:well-defined}
		Let $\cC$ be a locally small $F'/F$-linear Artinian abelian $\Gamma_{F'/F}$-category and $\DD(\cC)$ be an abelian full subcategory of $\DS(\cC)$ satisfying Condition \ref{con:subobject_inv}. Assume that $\Gamma_{F'/F}$ is finite, or that $\cC$ is Noetherian. Let $(M,\varphi_\sigma)$ be a simple descent datum.
		\begin{enumerate}
			\item The object $M\in\cC$ is nonzero and semisimple.
			\item The simple direct summands of $M$ are unique up to isomorphism and twists by $\Gamma_{F'/F}$.
			\item If $M'$ is a simple direct summand of $M$, so are its twists by $\Gamma_{F'/F}$.
		\end{enumerate}
	\end{lem}
	
	\begin{proof}
		Since a zero object in $\cC$ with a (unique) descent datum is a zero object of $\DD(\cC)$ (recall Proposition \ref{prop:(co)limit} (2)), $M$ is nonzero. Since $M$ is Artinian, we can and do choose a simple subobject $M'$ of $M$. For $\sigma\in\Gamma_{F'/F}$, we regard ${}^\sigma M'$ as a subobject of $M$ by ${}^\sigma M'\hookrightarrow {}^\sigma M\overset{\varphi_\sigma}{\cong} M$. Therefore it will suffice to show that $M$ is the sum of $M'$ and its finitely many Galois twists.
		
		For a finite subset $I\subset\Gamma_{F'/F}$, we set $M'_{I}=\sum_{\sigma\in I} {}^\sigma M'$. Thanks to our hypothesis, one can find a finite subset $I$ such that $M'_{I}=M'_{J}$ for any $J\supset I$. Then $M'_I$ is closed under formation of $\varphi_{\sigma}$ ($\sigma\in\Gamma_{F'/F}$). Therefore $(\varphi_\sigma)$ restricts to a descent system on $M'_I$. It lies in $\DD(\cC)$ by Condition \ref{con:subobject_inv}. Since $(M,\varphi_\sigma)$ is simple,
		we have $M'_I=M$. This completes the proof.
	\end{proof}

	\begin{proof}[Proof of Theorem \ref{thm:Loewy}]
		We can define a unique map
		\begin{equation}
			\Simple(\DD(\cC))\to\Gamma_{F'/F}\backslash \Simple(\cC)\label{eq:loewy_map}
		\end{equation}
		satisfying \ref{con:Loewy_bij} by assigning any simple subobject $M'$ of $M$ for each simple descent datum $(M,\varphi_\sigma)$ (Lemma \ref{lem:well-defined}).
		
		To see that \eqref{eq:loewy_map} is injective, let $(N,\psi_\sigma)$ be another simple descent datum such that $M'\cong {}^\sigma N'$ for some $\sigma\in\Gamma_{F'/F}$. Then choose a nonzero morphism $M\to M'$ and consider the inclusion map $N'\hookrightarrow N$ to get
		\[0\neq \Hom_{\cC}(M',{}^\sigma N')
		\hookrightarrow \Hom_{\cC}(M,N)
		\cong F'\otimes_F \Hom_{\DD(\cC)}((M,\varphi_\sigma),(N,\psi_\sigma))\]
		(use Condition \ref{con:bcthm} for the last isomorphism).
		Schur's lemma now implies an isomorphism
		$(M,\varphi_\sigma)\cong (N,\psi_\sigma)$.
		
		Finally, we prove that \eqref{eq:loewy_map} is surjective. Let $N$ be a simple object of $\cC$. Choose a simple descent datum $(M,\varphi_\sigma)$ and a nonzero morphism $M\to N$ (Condition \ref{con:surj}). Since $M$ is semisimple and $N$ is simple, $N$ is a direct summand of $M$. This shows that $N$ lies in the image.
	\end{proof}

	We note that computing the bijection of Theorem \ref{thm:Loewy} will be simplified if we are given a Loewy datum in the sense that it is essentially reduced to the finite case:
	
	\begin{thm}\label{thm:basic_Loewy_datum}
		Let $(\cC,\DD(\cC),\cC_{F''},F'\otimes_{F''}(-),\alpha^{F''}_\sigma)$ be a Loewy datum.
		\begin{enumerate}
			\item For every $F''\in\Lambda_{F'/F}$, $\cC_{F''}$ is a Loewy $\Gamma_{F''/F}$-category.
			\item Let $M\in\cC$ be a simple object, and $M_{F''}\in\cC_{F''}$ be an object with an isomorphism $F'\otimes_{F''} M_{F''}\cong M$. Then the simple descent datum corresponding to $M$ is obtained from that corresponding to $M_{F''}$ by Lemma \ref{lem:baseupds}.
		\end{enumerate}
	\end{thm}
	
	For its proof, let us record a general statement:
	
	\begin{lem}\label{lem:conservative}
		Let $F'/F$ be any extension of fields. Then any exact $F$-linear functor $F'\otimes_F(-)$
		from a locally small $F$-linear abelian category $\cC$ to a locally small $F'$-linear abelian category $\cD$ with the base change property of Hom spaces (i.e., the canonical map $F'\otimes_F \Hom_{\cC}(M,N)\to \Hom_{\cD}(F'\otimes_F M,F'\otimes_F N)$ is an isomorphism for every pair $M,N\in\cC$) reflects zero objects, isomorphisms, the finite length property, simple objects, and semisimple objects.
	\end{lem}
	
	The reflection of semisimplicity will be used in Section~\ref{sec:relation}.
	
	\begin{proof}
		The assertion for zero objects follows from the base change property. In fact, an object of a preadditive category with the zero endomorphism ring is zero. It is elementary that an exact functor of abelian categories reflects isomorphisms, the finite length property, and simple objects if it reflects zero objects. The assertion for semisimple objects is verified in a similar way to \cite[Lemma 3.20 (1)]{hayashisuper} through the induction by length.
	\end{proof}
	
	\begin{proof}[Proof of Theorem \ref{thm:basic_Loewy_datum}]
		Corollary \ref{cor:finite_loewy} and Lemma \ref{lem:conservative} imply (1) (consider the functor $F'\otimes_{F''}(-)$ for the finite length property). 
		
		Let $M$ and $M_{F''}$ be as in (2). Let $S_{F''}$ be the simple descent system corresponding to $M_{F''}$. Then $F'\otimes_{F''} S_{F''}$ is a simple descent datum by Theorem \ref{thm:Ld->Lp} and Condition \ref{con:ess_surj_Ld}. Moreover, we have
		\[\begin{split}
			0&\neq F'\otimes_{F''}\Hom_{\cC_{F''}}(\pi(S_{F''}),M_{F''})\\
			&\overset{\mathrm{\ref{con:bcthm_Ld}}}{\cong}
			\Hom_{\cC}(F'\otimes_{F''} \pi(S_{F''}),F'\otimes_{F''} M_{F''})\\
			&\cong \Hom_{\cC}(\pi(F'\otimes_{F''}S_{F''}),M)
		\end{split}
		\]
		since the forgetful functor $\pi$ commutes with $F'\otimes_{F''}(-)$. This shows that $F'\otimes_{F''}S_{F''}$ corresponds to $M$. This completes the proof of (2).
	\end{proof}

	We also give explicit computation on Lemma \ref{lem:well-defined}:
	
	\begin{prop}[\cite{MR1500635}]\label{prop:Loewy}
		Let $\cC$ be Loewy $\Gamma_{F'/F}$-category with $F'/F$ quadratic. Then a simple object $M$ of $\cC$ admits a descent datum if and only if $\Res_{F'/F} M$ is not simple. In particular, the simple descent datum corresponding to $M$ is $(M,\varphi_\sigma)$ (resp.~$\Res_{F'/F} M$) if $M$ admits (resp.~does not admit) a descent datum $(\varphi_\sigma)$.
	\end{prop}
	
	\begin{proof}
		If $M$ admits a descent datum $(\varphi_\sigma)$ then the unit $(\varphi^{-1}_\sigma):(M,\varphi_\sigma)\to\Res_{F'/F} M$ of Proposition \ref{prop:res} is not epic but monic.
		
		Conversely, suppose that $\Res_{F'/F} M$ is not simple. It follows by construction of $\Res_{F'/F} M$ that $\pi(\Res_{F'/F} M)$ is nonzero. In particular, $\Res_{F'/F} M$ is nonzero. We choose a nonzero proper subobject
		$(M_0,\varphi^0_\sigma)\subset\Res_{F'/F} M$. Since
		\[\pi(\Res_{F'/F} M)=M\oplus \bar{M},\]
		$M_0$ is isomorphic to $M$ or $\bar{M}$. Use $\varphi^0_\sigma$ if necessary to see that $M_0\cong M$. We now obtain a descent datum on $M$ by transferring that on $M_0$.
	\end{proof}

	Concerning existence of descent data, let us record another consequence of Lemma \ref{lem:well-defined}:
	
	\begin{prop}\label{prop:existence_DD}
		Consider the setting of Lemma \ref{lem:well-defined}. Assume that $(\cC,\DD(\cC))$ satisfies Conditions \ref{con:fin_length} and \ref{con:bcthm}. Then for a simple descent datum $(M,\varphi_\sigma)$, the following conditions are equivalent:
		\begin{enumerate}
			\item[(a)] $M$ is simple in $\cC$;
			\item[(b)] $\End_{\cC}(M)$ is a division algebra over $F$.
		\end{enumerate}
		Moreover, if all the endomorphisms of simple objects of $\cC$ are scalar, these are also equivalent to:
		\begin{enumerate}
			\item[(c)] $\End_{\cC}(M)=F'\id_{M}$;
			\item[(d)] $\End_{\DD(\cC)}((M,\varphi_\sigma))=F\id_{(M,\varphi_\sigma)}$.
		\end{enumerate}
	\end{prop}
	
	\begin{proof}
		The equivalence of (a) and (b) is immediate from Lemma \ref{lem:well-defined}. We obtain (c) $\iff$ (d) from Condition \ref{con:bcthm}. The implication (c) $\Rightarrow$ (b) is clear. The additional condition in the statement guarantees (a) $\Rightarrow$ (c). This completes the proof.
	\end{proof}
	
	\begin{ex}\label{ex:End}
		Let $F$ be an algebraically closed field, and $\cC$ be a locally small $F$-linear abelian category. If we are given a simple object $M\in\cC$ with $\End_{\cC}(M)$ finite-dimensional, then we have $\End_{\cC}(M)=F\id_M$. More generally, consider the following conditions:
		\begin{enumerate}
			\item[(a)] $\cC$ is locally finite-dimensional, i.e., $\Hom_{\cC}(M,N)$ is finite-dimensional for every pair $M,N$ of objects of $\cC$;
			\item[(b)] $\End_{\cC}(M)$ is of finite dimension for every simple object $M\in\cC$;
			\item[(c)] $\End_{\cC}(M)=F\id_M$ for every simple object $M\in\cC$.
		\end{enumerate} 
		It is clear that (a) $\Rightarrow$ (b) $\iff$ (c). The converse direction (b) $\Rightarrow$ (a) holds true (without the algebraically closed condition) if every object of $\cC$ is of finite length.
	\end{ex}
	
	We will revisit this result in the next section for the pseudo-absolute simplicity (Theorem \ref{thm:abs_simple}).

	\subsection{Existence of descent data}\label{sec:BTcocycle}
	
	In this section, we discuss descent data on simple objects. For arguments below, we work with slightly more general objects for broader applications (Example \ref{ex:line_bundle}).
	
	\begin{prop}\label{prop:dd_unique}
		Let $(\cC,\DD(\cC))$ be a pair of a locally small $F'/F$-linear abelian $\Gamma_{F'/F}$-category and an abelian full subcategory $\DD(\cC)\subset\DS(\cC)$ with Condition \ref{con:bcthm}. Let $M$ be an object of $\cC$ with $\End_{\cC}(M)$ being a division algebra. Then descent data on $M$ are at most unique up to isomorphism in $\DD(\cC)$.
	\end{prop}
	
	Typical examples of $M$ satisfying the condition are simple objects.
	
	\begin{proof}
		Let $(\varphi_\sigma)$ and $(\varphi'_\sigma)$ be descent data on $M$. Then 
		\[\Hom_{\DD(\cC)}((M,\varphi_\sigma),(M,\varphi'_\sigma))\]
		is nonzero from Condition \ref{con:bcthm}. Pick any nonzero element $f$. This is an isomorphism in $\cC$ by the hypothesis on $\End_{\cC}(M)$. The assertion now follows from Proposition \ref{prop:(co)limit} (1).
	\end{proof}

	\begin{cor}
		Let $(\cC,\DD(\cC))$ be a pair of a locally small and essentially small $F'/F$-linear abelian $\Gamma_{F'/F}$-category and an abelian full subcategory $\DD(\cC)\subset\DS(\cC)$ with Condition \ref{con:bcthm}.
		Then $(M,\varphi_\sigma)\mapsto M$ determines a well-defined bijection between the sets of simple descent data whose underlying objects in $\cC$ are simple and of simple objects of $\cC$ admitting descent data.
	\end{cor}
	
	To judge the simplicity of $M$, see Proposition \ref{prop:existence_DD}.

	We next wish to generalize the obstruction class $\beta$ to existence of descent data in \cite[Section~12]{MR207712}. For applications to the context of Theorem \ref{thm:Loewy} in the next section, we would like to work without the self-conjugacy hypothesis. For this, let us introduce a fundamental notion on rationality problems:
	
	\begin{defn}[Field of rationality]
		For an $F$-linear $\Gamma_{F'/F}$-category $\cC$ and $M\in\cC$, we set
		$\Gamma_M\coloneqq \{\sigma\in\Gamma_{F'/F}:~{}^\sigma M\cong M\}$ and $F(M)\coloneqq (F')^{\Gamma_M}$.
	\end{defn}
	
	Though this makes sense in general, we morally expect that $\End_{\cC}(M)$ is a division algebra and that $\Gamma_M$ is open in $\Gamma_{F'/F}$. These will be satisfied later in our applications.
	
	We now start to define the Borel--Tits cocycles in the finite setting. In this case, we give an indirect but canonical formulation using the bijection of \cite[Chapter X, \S5, Proposition 9]{MR554237}:
	
	\begin{defn-prop}[Borel--Tits cocycle]\label{def-prop:bt_finite}
		Suppose that $\Gamma_{F'/F}$ is finite. Let $\cC$ be a semilinear abelian $\Gamma_{F'/F}$-category. Consider an $M\in\cC$ satisfying the equality $\End_{\cC}(M)=F'\id_M$.
		Then $\End_{\DS(\cC;F(M)/F)}(\Res_{F'/F(M)} M)$ is a central simple algebra over $F(M)$ splitting over $F'$. We set $\beta^{\BT}_M\in H^2(\Gamma_M,(F')^\times)$ as the inverse of the corresponding cocycle\footnote{According to \cite[Chapter X, \S5, Remark]{MR554237}, there are two mutually opposite well-known correspondences. As mentioned above, we adopt Serre's here.}.
	\end{defn-prop}
	
	\begin{proof}
		This is immediate from Corollary \ref{cor:bc_criterion} and the construction of the underlying object of $\Res_{F'/F(M)} M$.
	\end{proof}
	
	\begin{rem}\label{rem:relation}
		We need the inverse for the compatibility with Borel--Tits' definition in \cite[Section~12]{MR207712}. In fact, let $G$ be a connected reductive algebraic group over a field $F$ of characteristic zero, and $V$ be a self-conjugate irreducible representation of $\bar{F}\otimes_F G$. Let $\beta_V$ be Borel--Tits' 2-cocycle attached to $V$, and $C$ be the corresponding central division algebra. Then $V$ descends to an absolutely irreducible $D$-representation of $G$ on $D^n$ for some $n\geq 1$ (\cite[12.6. Proposition. (iv)]{MR207712}, \cite[2.5. Th\'eor\`eme]{MR277536}). To get left $D$-linear maps respecting the action of $G$, we multiply $D$ from the right side. The opposite algebra appears in this way. Correspondingly, we need the inverse for the cocycle. We will see that our definition agrees with Borel--Tits' by Example \ref{ex:rep}, Proposition \ref{prop:key_computation_BT}, Theorem \ref{thm:End}, and \cite[12.6 and 12.7]{MR207712}.
	\end{rem}
	
	Let us perform a key computation to relate $\beta^{\BT}_M$ with the existence problem of descent systems:
	
	\begin{prop}\label{prop:key_computation_BT}
		Let $\cC$ be a semilinear abelian $\Gamma_{F'/F}$-category with $\Gamma_{F'/F}$ finite, and $M$ be a self-conjugate object of $\cC$ satisfying $\End_{\cC}(M)=F'\id_M$. Choose an isomorphism $\varphi_\sigma:{}^\sigma M\cong M$ for each $\sigma\in\Gamma_{F'/F}$. Write
		\[\varphi_\sigma\circ {}^\sigma\varphi_\tau\circ \mu_{\sigma,\tau,M}^{-1}\circ
		\varphi^{-1}_{\sigma\tau}
		=\beta(\sigma,\tau)\id_M\]
		for $\beta(\sigma,\tau)\in F'$ ($\sigma,\tau\in\Gamma_{F'/F}$). Then $\beta$ is a 2-cocycle. Moreover, we have $\beta^{\BT}_M=\beta$ in $H^2(\Gamma_{F'/F},(F')^\times)$. In particular, $\beta$ only depends on the isomorphism class of $M$ as a cohomology class.
	\end{prop}
	
	For generalities on correspondence of second Galois cohomology classes and central simple algebras, see \cite[Chapter X, \S5]{MR554237}. \cite[Construction 4.2]{hayashisuper} is available for its brief exposition to the extent necessary right below. 
	
	\begin{proof}
		To save space, we write $\mu_{\sigma,\tau,M}=\mu_{\sigma,\tau}$ for $\sigma,\tau\in\Gamma_{F'/F}$.
		
		For $\sigma\in\Gamma_{F'/F}$, consider the diagram
		\begin{equation}
			\begin{tikzcd}
				F'\otimes_F \End_{\DS(\cC)} (\Res_{F'/F} M)\ar[r, "\mathrm{\ref{con:bcthm}}", "\sim"']
				\ar[d, "\sigma"']
				&\End_{\cC}\left(\prod_{\tau\in\Gamma_{F'/F}} {}^\tau M\right)
				\ar[d, dashed]\ar[r, "(\varphi_\tau)"]
				&\End_{\cC}\left(\prod_{\tau\in\Gamma_{F'/F}} M\right)\ar[d, dashed]\\
				F'\otimes_F \End_{\DS(\cC)} (\Res_{F'/F} M)\ar[r, "\mathrm{\ref{con:bcthm}}", "\sim"']
				&\End_{\cC}\left(\prod_{\tau\in\Gamma_{F'/F}} {}^\tau M\right)\ar[r, "(\varphi_\tau)"]
				&\End_{\cC}\left(\prod_{\tau\in\Gamma_{F'/F}} M\right).
			\end{tikzcd}\label{diag:Galois_action}
		\end{equation}
		The left vertical arrow is given by the action of $\sigma$ on $F'$. Under the identification of $\End_{\cC}\left(\prod_{\tau\in\Gamma_{F'/F}} {}^\tau M\right)$ with $\prod_{\tau_1,\tau_2\in\Gamma_{F'/F}}\Hom_{\cC}({}^{\tau_1}M,{}^{\tau_2}M)$, the left square is commutative for $(f_{\tau_1,\tau_2})\mapsto (\mu_{\sigma,\sigma^{-1}\tau_2}\circ{}^\sigma f_{\sigma^{-1}\tau_1,\sigma^{-1}\tau_2}\circ \mu^{-1}_{\sigma,\sigma^{-1}\tau_1})$ since
		\[\mu_{\sigma,\sigma^{-1}\tau_2}\circ {}^\sigma (cf_{\sigma^{-1}\tau_1,\sigma^{-1}\tau_2}) \circ \mu^{-1}_{\sigma,\sigma^{-1}\tau_1}
		=\sigma(c)\mu_{\sigma,\sigma^{-1}\tau_2}\circ {}^\sigma f_{\sigma^{-1}\tau_1,\sigma^{-1}\tau_2}\circ \mu^{-1}_{\sigma,\sigma^{-1}\tau_1}
		=\sigma(c)f_{\tau_1,\tau_2}
		\]
		for
		$c\in F'$ and $f=(f_{\tau_1,\tau_2}:{}^{\tau_1}M\to {}^{\tau_2}M)\in \End_{\DS(\cC)} (\Res_{F'/F} M)$ by definition of $\End_{\DS(\cC)} (\Res_{F'/F} M)$.
		It is straightforward that the right vertical arrow which makes the right square commutative is given by
		\[(c_{\tau_1,\tau_2}\id_M)\mapsto(\sigma(c_{\sigma^{-1}\tau_1,\sigma^{-1}\tau_2})\varphi_{\tau_2}\circ \mu_{\sigma,\sigma^{-1}\tau_2}\circ {}^\sigma\varphi^{-1}_{\sigma^{-1}\tau_2}\circ{}^\sigma\varphi_{\sigma^{-1}\tau_1}\circ \mu^{-1}_{\sigma,\sigma^{-1}\tau_1}\circ \varphi^{-1}_{\tau_1}).\]
		Its composition from the right side with the entry-wise action of $\sigma$ on the scalars is
		\begin{equation}
			(c_{\tau_1,\tau_2}\id_M)\mapsto(c_{\sigma^{-1}\tau_1,\sigma^{-1}\tau_2}\varphi_{\tau_2}\circ \mu_{\sigma,\sigma^{-1}\tau_2}\circ {}^\sigma\varphi^{-1}_{\sigma^{-1}\tau_2}\circ{}^\sigma\varphi_{\sigma^{-1}\tau_1}\circ \mu^{-1}_{\sigma,\sigma^{-1}\tau_1}\circ \varphi^{-1}_{\tau_1}).\label{eq:Galois_action}
		\end{equation}

		Define $w_\sigma=((w_\sigma)_{\tau,\rho})\in \End_{\cC}(\prod_{\sigma\in\Gamma_{F'/F}} M)$ by
		\[(w_\sigma)_{\tau,\rho}=\begin{cases}
			\id_M&(\tau=\sigma\rho)\\
			0&(\mathrm{otherwise}).
		\end{cases}\]
		Then we have 
		\[w_\sigma\circ (c_{\tau_1,\tau_2}\id_M)\circ w^{-1}_\sigma =(c_{\sigma^{-1}\tau_1,\sigma^{-1}\tau_2}\id_M)\]
		for $\sigma\in\Gamma_{F'/F}$ and $f=(c_{\tau_1,\tau_2})\in \End_{\cC}(\prod_{\sigma\in\Gamma_{F'/F}} M)$, and $w_{\sigma}\circ w_\tau=w_{\sigma\tau}$ for elements $\sigma,\tau\in\Gamma_{F'/F}$. Therefore
		\eqref{eq:Galois_action} coincides with the conjugation of
		\[\diag(\varphi_{\tau} \circ \mu_{\sigma,\sigma^{-1}\tau}\circ {}^{\sigma}\varphi^{-1}_{\sigma^{-1}\tau}\circ \varphi^{-1}_{\sigma})\circ w_{\sigma}.\]

		It is now straightforward to check that the 2-cocycle attached to the composite isomorphism $F'\otimes_F \End_{\DS(\cC)} (\Res_{F'/F} M)
		\cong \End_{\cC}\left(\prod_{\tau\in\Gamma_{F'/F}} M\right)$ in the upper horizontal sequence of \eqref{diag:Galois_action} agrees with $\beta$ (use the elementary fact that if we are given a scalar map $f\in \End_{\cC}(M)$, say $f=c\id_M$ ($c\in F'$) and an isomorphism $g:M\cong N$, then we have $g\circ f\circ g^{-1}=c\id_N$).
	\end{proof}
	
	A similar argument to \cite[Theorem 2.2.5]{MR4627704} now implies
	
	\begin{cor}\label{cor:BT}
		Consider the setting of Definition-Proposition \ref{def-prop:bt_finite}.
		Then $M$ admits a descent system if and only if $M$ is self-conjugate and $\beta^{\mathrm{BT}}_M$ is trivial in $H^2(\Gamma_{F'/F},(F')^\times)$.
	\end{cor}

	\begin{ex}[{\cite[Section~VIII]{E1914}}]
		Put $F=\bR$ and $F'=\bC$. Write $\sigma\in\Gamma_{\bC/\bR}$ for the nontrivial element. Let $\cC$ be a semilinear abelian $\Gamma_{\bC/\bR}$-category with Condition \ref{con:bcthm}, and $M\in\cC$ be a simple object with $\End_{\cC}(M)=\bC\id_M$.
		Identify $H^2(\Gamma_{\bC/\bR},\bC^\times)$ with $\{\pm 1\}$. The sign corresponding to $\beta^{\BT}_M$ is called the index. We denote it by $\Index(M)$. Explicitly, this is given as follows: choose $\varphi:\bar{M}\cong M$. Then $\varphi\circ \bar{\varphi}\circ \mu^{-1}_{M,\sigma,\sigma}\circ u^{-1}_M$ is a nonzero real scalar map. The index is the sign of its scalar.
	\end{ex}

	We obtain a generalization of the full statement of the classical work \cite{MR1500635}:
	
	\begin{cor}
		Put $F'/F=\bC/\bR$. Let $\cC$ be a Loewy $\Gamma_{\bC/\bR}$-category, and $M$ be a simple object of $\cC$ with $\End_{\cC}(M)=\bC\id_M$.
		\begin{enumerate}
			\item The object $M$ admits a descent system $(\varphi_\sigma)$ if and only if $M$ is self-conjugate and $\Index (M)=1$. Moreover, if these equivalent conditions are satisfied, then $\End_{\DS(\cC)}((M,\varphi_\sigma))=\bR \id_{(M,\varphi_\sigma)}$.
			\item Suppose that $M$ does not admit any descent systems. Then the simple descent system corresponding to $M$ is $\Res_{\bC/\bR} M$. Moreover, we have
			\[\End_{\DS(\cC)}(\Res_{\bC/\bR} M)\cong\begin{cases}
				\bC&\mathrm{if}~\bar{M}\not\cong M;\\
				\bH&\mathrm{if}~\bar{M}\cong M.
			\end{cases}\]
		\end{enumerate}
	\end{cor}
	
	We will discuss its generalization in the next section. Here we give a simple argument, based on the fact that the finite-dimensional real division algebras are $\bR$, $\bC$, and $\bH$ (the Frobenius theorem).
	
	\begin{proof}
		The first part of (1) is a restatement of Proposition \ref{prop:key_computation_BT}. The equality $\End_{\DS(\cC)}((M,\varphi_\sigma))=\bR \id_{(M,\varphi_\sigma)}$ in (1) follows from Proposition \ref{prop:existence_DD}.
		
		For the first part of (2), see Proposition \ref{prop:Loewy}. We can determine the real division algebra $\End_{\DS(\cC)}(\Res_{\bC/\bR} M)$ by the isomorphism
		\[\bC\otimes_\bR \End_{\DS(\cC)}(\Res_{\bC/\bR} M)\cong \End_{\cC}(M\oplus \bar{M})\]
		and seeing the dimension of the right hand side.
	\end{proof}

	Though the condition $\End_{\cC}(M)=\bC\id_M$ is reasonable and expected to be satisfied in applications, one can easily find an easy counterexample to this since the finiteness conditions in the definition of Loewy $\Gamma_{\bC/\bR}$-categories are just as abelian categories:
	
	\begin{ex}
		Put $F'/F=\bC/\bR$. Let $\cC$ be the category of finite-dimensional vector spaces over the field $\bC(t)$ of complex rational functions of one variable. Then $\cC$ is a Loewy $\Gamma_{\bC/\bR}$-category for the complex conjugation. This is a semisimple abelian category. Moreover, the only simple object is $\bC(t)$ (up to isomorphisms), whose $\bC$-algebra of endomorphisms is $\bC(t)$.
	\end{ex}

	We next discuss the infinite case. We wish to introduce a cohomology class to the obstruction of existence of a descent datum on a semilinear abelian $\Gamma_{F'/F}$-category with Condition \ref{con:bcthm}. The self conjugacy is an obvious necessary condition. 
	
	One could then define $\beta^{\BT}_M$ by the numerical description in Proposition \ref{prop:key_computation_BT} if $\End_{\cC}(M)=F'\id_M$. However, $\beta^{\BT}_M$ is morally expected to be continuous. The cocycle $\beta^{\BT}_M$ should then arise from a descent datum in $\cC_{F''}$ for certain $F''\in\Lambda_{F'/F}$ in the setting of a Loewy datum. We thus need a functor $F''\otimes_F(-)$. If one wishes to see independence of choice of $F''$ for the cocycle $\beta$ in $H^2(\Gamma_{F'/F},(F')^\times)$, the cocycles for fields in $\Lambda_{F'/F}$ should be compared in an ambient field in $\Lambda_{F'/F}$. We need additional structures to guarantee the compatibility of the construction of $\beta$.
	
	\begin{defn}\label{defn:BTdatum}
		\begin{enumerate}
			\item A Borel--Tits datum consists of
			\begin{itemize}
				\item a locally small semilinear abelian $\Gamma_{F'/F}$-category $\cC=\cC_{F'}$,
				\item an abelian full subcategory $\DD(\cC)\subset\DS(\cC)$,
				\item a locally small semilinear abelian $\Gamma_{F''/F}$-category $\cC_{F''}$
				for each $F''\in\Lambda_{F'/F}$,
				\item an $F''$-linear functor $F'\otimes_{F''}(-):\cC_{F''}\to \cC$
				for $F''\in\Lambda_{F'/F}$,
				\item an $F'''$-linear functor $F''\otimes_{F'''}(-):\cC_{F'''}\to \cC_{F''}$
				for $F'',F'''\in\Lambda_{F'/F}$ with $F'''\subset F''$,
				\item a natural transformation $\alpha^{F''}_{\sigma}:{}^\sigma(F'\otimes_{F''}(-))\cong F'\otimes_{F''} {}^{\sigma|_{F''}}(-)$ for each field $F''\in\Lambda_{F'/F}$ and $\sigma\in\Gamma_{F'/F}$,
				\item a natural transformation $\alpha^{F''/F'''}_\sigma:
				{}^\sigma (F''\otimes_{F'''}(-))\cong F''\otimes_{F'''} {}^{\sigma|_{F'''}}(-)$
				for $F'',F'''\in\Lambda_{F'/F}$ with $F'''\subset F''$ and $\sigma\in\Gamma_{F''/F}$,
				\item a natural transformation
				\[\gamma^{F^{(2)}/F^{(3)}/F^{(4)}}:
				F^{(2)}\otimes_{F^{(3)}}(F^{(3)}\otimes_{F^{(4)}}(-))\cong F^{(2)}\otimes_{F^{(4)}}(-)\]
				for $F^{(i)}\in\Lambda_{F'/F}\cup\{F'\}$ ($2\leq i\leq 4$) with $F^{(i)}\subset F^{(i-1)}$ for $i=3,4$,
			\end{itemize}
			such that
			\begin{enumerate}[label=(BT\arabic*)]
				\item \label{con:strict_BT} $\DD(\cC)\subset\DS(\cC)$ is strict, i.e., closed under formation of isomorphisms in $\DS(\cC)$,
				\item \label{con:exact_BT} for $F''\in\Lambda_{F'/F}$, the functor $F'\otimes_{F''}(-)$ is exact,
				\item \label{con:bcthm_BT} for $F''\in\Lambda_{F'/F}$ and $M,N\in\cC_{F''}$, the canonical map
				\[F'\otimes_{F''} \Hom_{\cC_{F''}}(M,N)
				\to \Hom_{\cC}(F'\otimes_{F''}M,F'\otimes_{F''}N)\]
				is an isomorphism,
				\item \label{con:associative_BT} $\alpha^{F''}_\sigma$ and $\alpha^{F''/F'''}_\sigma$ are associative in $\sigma$,
				\item \label{con:natural_BT} $\gamma$ are natural in $F^{(\bullet)}$, i.e., for
				\[\begin{array}{cc}
					F^{(i)}\in\Lambda_{F'/F}\cup\{F'\}&(2\leq i\leq 5)
				\end{array}\]
				with $F^{(i)}\subset F^{(i-1)}$ for $3\leq i\leq 5$ and $M\in\cC_{F^{(5)}}$, we have
				\[\gamma^{F^{(2)}/F^{(4)}/F^{(5)}}_M\circ \gamma^{F^{(2)}/F^{(3)}/F^{(4)}}_{F^{(4)\otimes_{F^{(5)}}}M}
				=\gamma^{F^{(2)}/F^{(3)}/F^{(5)}}_M\circ (F^{(2)}\otimes_{F^{(3)}}\gamma^{F^{(3)}/F^{(4)}/F^{(5)}}_{M}),
				\]
				\item \label{con:esssurj_BT} the functor
				\begin{equation}
					\DS(\cC_{F''})\to \DS(\cC)\label{eq:bc_of_ds}
				\end{equation}
				of Lemma \ref{lem:baseupds} is essentially surjective onto $\DD(\cC)$, 
				\item \label{con:rationality} for every object $M\in\cC$, there exist $F''\in\Lambda_{F'/F}$ and $M_{F''}\in\cC_{F''}$ with an isomorphism
				$F'\otimes_{F''} M_{F''}\cong M$. We will call $M_{F''}$ an $F''$-form of $M$.
			\end{enumerate}
			If there is no risk of confusion, we say that a Borel--Tits datum on a semilinear abelian $\Gamma_{F'/F}$-category $\cC$ is given. In this case, we follow the above for the symbols.
			\item We say a Borel--Tits datum on a semilinear abelian $\Gamma_{F'/F}$-category $\cC$ is \emph{strong} if
			\begin{enumerate}[label=(SBT)]
				\item \label{con:Schur} $\End_{\cC}(M)=F'\id_M$ for every simple object $M\in\cC$.
			\end{enumerate}
		\end{enumerate}
	\end{defn}
	
	\begin{rem}\label{rem:Galois_descent}
		Consider a Borel--Tits datum on a semilinear abelian $\Gamma_{F'/F}$-category $\cC$ as above. Then for $F''\in\Lambda_{F'/F}$, the proof of Theorem \ref{thm:Ld->Lp} (1) implies that the functor \eqref{eq:bc_of_ds} is fully faithful (see Remark \ref{rem:overcondition}). Hence the functor \eqref{eq:bc_of_ds} in \ref{con:esssurj_BT} determines an equivalence of categories.
	\end{rem}
	
	\begin{rem}
		It is morally natural to assume the compatibility of $\alpha$ and $\gamma$, i.e., if we are given $F'',F'''\in\Lambda_{F'/F}$ with $F'''\subset F''$ and $\sigma\in\Gamma_{F'/F}$, the diagram
		\[\begin{tikzcd}
			{}^\sigma (F'\otimes_{F''} (F''\otimes_{F'''} (-)))
			\ar[rrr, "\id_{{}^\sigma(-)}\circ\gamma^{F'/F''/F'''}", "\sim"']
			\ar[d, "\alpha^{F''}_\sigma\circ \id_{F''\otimes_{F'''}(-)}"', "\sim"{sloped, above}]
			&&&{}^\sigma (F'\otimes_{F'''}(-))
			\ar[dd, "\alpha^{F'''}_{\sigma}", "\sim"{sloped, below}]\\
			F'\otimes_{F''} {}^{\sigma|_{F''}}(F''\otimes_{F'''} (-))
			\ar[d, "\id_{F'\otimes_{F''}(-)}\circ\alpha^{F''/F'''}"', "\sim"{sloped, above}]\\
			F'\otimes_{F''} (F''\otimes_{F'''} {}^{\sigma|_{F'''}}(-))
			\ar[rrr, "\gamma^{F'/F''/F'''}\circ \id_{{}^{\sigma|_{F'''}}(-)}", "\sim"']
			&&&F'\otimes_{F'''} {}^{\sigma|_{F'''}}(-)
		\end{tikzcd}\]
		commutes. In this case, the composite functor
		\[F'\otimes_{F''} (F''\otimes_{F'''}(-)):\DS(\cC_{F'''})\to \DS(\cC_{F''})\to \DD(\cC)\]
		is isomorphic to $F'\otimes_{F'''}(-)$. This will reduce the essential surjectivity of \eqref{con:esssurj_BT} to the case $F''=F$ (cf.~Example \ref{ex:Gamma={e}}). More cancellation structures of $F''$ and the compatibility with Galois twists should be also assumed. The lack of these hypotheses will be harmless in the present paper by virtue of Proposition \ref{prop:dd_unique}.
	\end{rem}
	
	\begin{rem}\label{rem:F(M)_finite}
		For every object $M$ of a semilinear abelian $\Gamma_{F'/F}$-category $\cC$ with a Borel--Tits datum, $\Gamma_M$ is open in $\Gamma_{F'/F}$ by Condition \ref{con:rationality}. In particular, $F'/F(M)$ is a Galois extension of Galois group $\Gamma_M$.
	\end{rem}
	
	\begin{ex}[{\cite[Section~12]{MR207712}}]\label{ex:affinegroupscheme}
		Let $G$ be an affine group scheme over $F$. 
		Then $\Rep(F''\otimes_F G)$ with $F''\in\Lambda_{F'/F}\cup\{F'\}$ and $\DD(\Rep(F'\otimes_FG))$ form a Borel--Tits datum in the standard manner (note: \ref{con:rationality} is verified in a similar way to \cite[Proposition 3.33]{hayashisuper}). This is strong if
		\begin{enumerate}
			\item[(i)] $F$ is perfect and $F'=\bar{F}$, or
			\item[(ii)] $F'\otimes_FG$ is a split connected reductive algebraic group over $F'$.
		\end{enumerate}
	\end{ex}
	
	\begin{ex}[\cite{MR4627704}]\label{ex:line_bundle}
		Let $X$ be an algebraic variety over $F$ with an action of a linear algebraic group $G$. For each $F''\in\Lambda_{F'/F}\cup\{F'\}$, let
		\[\Coh(F''\otimes_F X,F''\otimes_F G)\]
		denote the category of $F''\otimes_F G$-equivariant coherent $\cO_{F''\otimes_F X}$-modules. Let
		\[\DD(\Coh(F'\otimes_FX,F'\otimes_F G))\]
		be the category of descent data in the usual sense. They form a Borel--Tits datum in the standard manner. Any $F''\otimes_F G$-equivariant line bundle $\cL$ on $F''\otimes_FX$ satisfies $\End_{\Coh(F'\otimes_F X,F'\otimes_F G)}(\cL)=F'\id_\cL$ if one of the following conditions is satisfied:
		\begin{enumerate}
			\item[(i)] $X$ is proper, geometrically reduced, and geometrically connected over $F$;
			\item[(ii)] $X$ is homogeneous, i.e., $\bar{F}\otimes_FG$ acts transitively on $\bar{F}\otimes_FX$.
		\end{enumerate}
	\end{ex}

	Suppose that we are given a semilinear abelian $\Gamma_{F'/F}$-category $\cC$ with a Borel--Tits datum. We see that our idea on the Borel--Tits cocycle works in this setting. 
	
	\begin{lem}\label{lem:translation_isom}
		Consider objects $M,N\in\cC$ satisfying $M\cong N$. If $\End_{\cC}(M)$ is a division algebra, every nonzero map $M\to N$ is an isomorphism.
	\end{lem}
	
	\begin{proof}
		Compose the inverse of the given isomorphism $M\cong N$, we may assume that $N=M$. Then the assertion follows from the hypothesis.
	\end{proof}
	
	\begin{lem}\label{lem:Galois_descent_of_self-conjugacy}
		Let $F''\in\Lambda_{F'/F}$, and $M\in\cC_{F''}$ with $\End_{\cC_{F''}}(M)=F''\id_M$ (equivalently, $\End_{\cC}(F'\otimes_{F''}M)=F'\id_{F'\otimes_{F''}M}$ by Condition \ref{con:bcthm_BT}). Then $M$ is self-conjugate if and only if so is $F'\otimes_{F''} M$.
	\end{lem}
	
	\begin{proof}
		The ``only if'' direction follows from existence of the isomorphisms $\alpha^{F''}_\sigma$. For the converse, assume that $F'\otimes_{F''} M$ is self-conjugate. Consider an element $\bar{\sigma}\in\Gamma_{F''/F}$. Fix its lift $\sigma\in\Gamma_{F'/F}$. Then consider an isomorphism
		\[\begin{split}
			F'\otimes_{F''} \Hom_{\cC_{F''}}({}^{\bar{\sigma}} M,M)
			&\overset{\mathrm{\ref{con:bcthm_BT}}}{\cong}
			\Hom_{\cC}(F'\otimes_{F''} {}^{\bar{\sigma}} M,F'\otimes_{F''} M)\\
			&\overset{\alpha^{F''}_{\sigma,M}}{\cong}
			\Hom_{\cC}({}^\sigma (F'\otimes_{F''} M),F'\otimes_{F''} M)\\
			&\neq 0.
		\end{split}
		\]
		Hence we can find a nonzero morphism $\varphi_{\bar{\sigma}}:{}^{\bar{\sigma}} M_{F''}\to M_{F''}$. In view of Lemma \ref{lem:translation_isom}, $F'\otimes_{F''} \varphi_{\bar{\sigma}}$ is an isomorphism since this is nonzero. We conclude that $\varphi_{\bar{\sigma}}$ is an isomorphism by 
		Lemma \ref{lem:conservative}.
		
	\end{proof}

	\begin{cons}[Borel--Tits cocycle]\label{cons:BT}
		Let $M\in\cC$ be an object satisfying the equality $\End_{\cC}(M)=F'\id_{M}$. Choose $F''\in\Lambda_{F'/F}$ and an $F''$-form $M_{F''}\in\cC_{F''}$ of $M$ (Condition \ref{con:rationality}). In virtue of the existence of $\alpha^{F''}_\sigma$ for $\sigma\in\Gamma_{F'/F}$, we have
		\[\Gamma_{F'/F''}\subset \Gamma_M=\Gamma_{F'/F(M)}.\]
		According to Lemmas \ref{lem:self-conjugacy} and \ref{lem:Galois_descent_of_self-conjugacy}, $M_{F''}$ is self-conjugate with respect to the twists by $\Gamma_{F''/F(M)}$. We now define $\beta^{\BT}_M\in H^2(\Gamma_M,(F')^\times)$ as the image of
		\[\beta^{\BT}_{M_{F''}}\in H^2(\Gamma_{F''/F(M)},(F'')^\times).\]
	\end{cons}

	\begin{defn-prop}\label{defprop:BT}
		\begin{enumerate}
			\item The cocycle $\beta^{\BT}_M$ is independent of choice of $F''$ and $M_{F''}$.
			\item The object $M$ admits a descent datum if and only if $M$ is self-conjugate and $\beta^{\BT}_M$ is trivial in $H^2(\Gamma_{F'/F},(F')^\times)$.
		\end{enumerate}	
	\end{defn-prop}
	
	\begin{proof}
		For (1), suppose that we are given an $F'''$-form $M_{F'''}$ and an $F''''$-form $M_{F''''}$ of $M$ for some finite Galois extensions $F''',F''''\in\Lambda_{F'/F}$.
		Pick $F''\in\Lambda_{F'/F}$ with $F''',F''''\subset F''$. Then $F''\otimes_{F'''} M_{F'''}$ and $F''\otimes_{F''''} M_{F''''}$ are $F''$-forms of $M$ (use $\gamma$). Moreover, there is an isomorphism $F''\otimes_{F'''} M_{F'''}\cong F''\otimes_{F''''} M_{F''''}$ by a similar argument to Lemma \ref{lem:Galois_descent_of_self-conjugacy}.
		
		Fix an isomorphism $\varphi_{\bar{\sigma}}:{}^{\bar{\sigma}} M_{F'''}\cong M_{F'''}$ for each $\bar{\sigma}\in\Gamma_{F'''/F(M)}$. For $\sigma\in\Gamma_{F''/F(M)}$, let $\varphi_\sigma:{}^\sigma (F''\otimes_{F'''}M_{F'''} )\cong F''\otimes_{F'''}M_{F'''}$ be the isomorphism of Construction \ref{cons:basechange}. Then the associativity of $\alpha^{F''/F'''}$ (Condition \ref{con:associative_BT}) implies
		\[\varphi_\sigma\circ {}^\sigma\varphi_\tau\circ \mu_{\sigma,\tau,F''\otimes_{F'''}M_{F'''}}\circ \varphi^{-1}_{\sigma\tau}
		=F''\otimes_{F'''} (\varphi_{\bar{\sigma}}\circ{}^{\bar{\sigma}}
		\varphi_{\bar{\tau}}\circ\mu_{{\bar{\sigma}},{\bar{\tau}},M_{F'''}}
		\circ \varphi^{-1}_{{\bar{\sigma}}{\bar{\tau}}}).\]
		Hence the image of $\beta^{\BT}_{M_{F'''}}$ in $H^2(\Gamma_{F''/F(M)},(F'')^\times)$ agrees with $\beta^{\BT}_{F''\otimes_{F'''} M_{F'''}}$ by virtue of Proposition \ref{prop:key_computation_BT}. Similarly, the image of $\beta^{\BT}_{M_{F''''}}$ in $H^2(\Gamma_{F''/F(M)},(F'')^\times)$ agrees with $\beta^{\BT}_{F''\otimes_{F''''} M_{F''''}}$. Finally, Proposition \ref{prop:key_computation_BT} implies
		\[\beta^{\BT}_{F''\otimes_{F'''} M_{F'''}}=\beta^{\BT}_{F''\otimes_{F''''} M_{F''''}}.\] 
		Part (1) now follows since $\beta^{\BT}_M$ is their image in $H^2(\Gamma_M,(F')^\times)$.		
		
		We next verify (2). We may and do assume that $M$ is self-conjugate. In particular, we have $\Gamma_M=\Gamma_{F'/F}$.
		If $M$ admits a descent datum then $M$ admits an $F$-form $M_F$ by virtue of Condition \ref{con:esssurj_BT}. It is then clear that $\beta^{\BT}_{M_F}$ and therefore $\beta^{\BT}_M$ are trivial. Conversely, suppose that $\beta^{\BT}_M$ is trivial. Fix $F''\in\Lambda_{F'/F}$ and an $F''$-form $M_{F''}$ as in Construction \ref{cons:BT}. Recall that the canonical map $H^2(\Gamma_{F''/F},(F'')^\times)\to H^2(\Gamma_{F'/F},(F')^\times)$ is injective from the fundamental long exact sequence of Galois cohomology groups and Hilbert's theorem 90 (see \cite[Chapter X, \S4, Proposition 6]{MR554237}). Therefore $\beta^{\BT}_{M_{F''}}$ is trivial. Corollary \ref{cor:BT} implies that $M_{F''}$ admits a descent system. Apply $F'\otimes_{F''}(-)$ to obtain a descent datum on $M$ (Condition \ref{con:esssurj_BT}). This completes the proof.
	\end{proof}

	In the rest of this small section, let $\cC$ be a semilinear abelian $\Gamma_{F'/F}$-category with a Borel--Tits datum. Concerning the pseudo-absolute simplicity (cf.~\cite[Definition-Proposition 3.30]{hayashisuper}), we develop its generalities for $\cC$.

	\begin{lem}\label{lem:pseudosimple_finite}
	Let $F''\in\Lambda_{F'/F}$.
	\begin{enumerate}
		 \item For $M,N\in\cC_F$, the natural map
		\[F''\otimes_F \Hom_{\cC_F}(M,N)\to \Hom_{\cC_{F''}}(F''\otimes_F M,F''\otimes_F N)\]
		is an isomorphism.
		\item The functor $F''\otimes_{F}(-):\cC_F\to\cC_{F''}$ determines an equivalence 
		\[\cC_{F,\mathrm{fl}}\simeq \DS(\cC_{F'',\mathrm{fl}})\]
		of categories.
		\item The functor $F''\otimes_{F}(-)$ respects semisimple objects.
	\end{enumerate}
	\end{lem}
	
	\begin{proof}
		Part (1) is immediate from \ref{con:bcthm_BT} and the faithfully flat descent (apply the base change $F'\otimes_{F''}(-)$ to the map of (1)).
		
		For (2), consider the functor $F''\otimes_F(-):\cC_{F}\simeq \DS(\cC_{F''})$ of Lemma \ref{lem:baseupds}. This admits a quasi-inverse by (1) and Remark \ref{rem:overcondition}. We obtain the equivalence by restricting this by virtue of Corollary \ref{cor:fl}.
		
		Part (3) is a consequence of (2) and Lemma \ref{lem:well-defined}.
	\end{proof}
	
	\begin{thm}\label{thm:abs_simple}
		For a simple object $M$ of $\cC_F$, the following conditions are equivalent:
		\begin{enumerate}
			\renewcommand{\labelenumi}{(\alph{enumi})}
			\item $F'\otimes_F M$ is simple;
			\item $F''\otimes_FM$ is simple for every $F''\in\Lambda_{F'/F}$;
			\item $\End_{\cC} (F'\otimes_F M)$ is a division algebra over $F'$;
			\item $\End_{\cC_{F''}} (F''\otimes_F M)$ is a division algebra over $F''$ for every $F''\in\Lambda_{F'/F}$.
		\end{enumerate}
		Moreover, if the Borel--Tits datum is strong then these are also equivalent to:
		\begin{enumerate}
			\item[(e)] $\End_{\cC} (F'\otimes_F M)=F'\id_{F''\otimes_F M}$;
			\item[(f)] $\End_{\cC_{F''}} (F''\otimes_F M)=F''\id_{F''\otimes_F M}$ for every $F''\in\Lambda_{F'/F}$;
			\item[(g)] $\End_{\cC_F}(M)=F\id_M$.
		\end{enumerate}
	\end{thm}
	
	\begin{proof}
		Schur's lemma implies (a) $\Rightarrow$ (c). We have (d) $\Rightarrow$ (b) from Lemma \ref{lem:pseudosimple_finite} (3).
		
		Assume that (c) is satisfied. Let $F''\in\Lambda_{F'/F}$. Pick any nonzero element $f\in \End_{\cC_{F''}} (F''\otimes_F M)$. We identify $F'\otimes_{F''} f$ with an element of $\End_{\cC}(F'\otimes_F M)$ by conjugation of $\gamma^{F'/F''/F}_M$. This is nonzero from \ref{con:bcthm_BT}. We now conclude that $f$ is an isomorphism by (c), \ref{con:bcthm_BT}, and Lemma \ref{lem:conservative}. This shows (c) $\Rightarrow$ (d).
		
		Assume that (b) is satisfied. We wish to prove (a). Let $N'\subset F'\otimes_F M$ be a nonzero subobject. Let $i$ be the corresponding monomorphism. One can find $F''\in\Lambda_{F'/F}$ and an $F''$-form $N'_{F''}$ of $N'$ by Condition \ref{con:rationality}. Consider the isomorphism
		\[F'\otimes_{F''}\Hom_{\cC_{F''}}(N'_{F''},F'\otimes_F M)
		\cong \Hom_{\cC}(N',F'\otimes_F M)
		\]
of \ref{con:bcthm_BT} (use the identification map $F'\otimes_{F''}N'_{F''}\cong N'$ and $\gamma^{F'/F''/F}_M$). Then the preimage of $i\in \Hom_{\cC}(N',F'\otimes_F M)$ lies in $F'''\otimes_{F''}\Hom_{\cC_{F''}}(N'_{F''},F'\otimes_F M)$ for certain $F'''\in\Lambda_{F'/F''}$. In virtue of Condition \ref{con:natural_BT} and Lemma \ref{lem:pseudosimple_finite} (1), we may replace $F''$ with $F'''$ to assume that $i$ descends to $i_{F''}:M_{F''}\to F''\otimes_F M$. This is a monomorphism by Lemma \ref{lem:conservative} (see the kernel). We conclude that $i$ is an isomorphism by Schur's lemma and the assumption (b). This completes the proof of the first part.

		To prove the latter part, assume that the given Borel--Tits datum is strong. Then Proposition \ref{prop:existence_DD} shows (a) $\iff$ (e) $\iff$ (g). The implications (g) $\Rightarrow$ (f) $\Rightarrow$ (e) follow from Lemma \ref{lem:pseudosimple_finite} (1) and \ref{con:bcthm_BT}. This completes the proof.
	\end{proof}
	
	\begin{cor}\label{cor:semisimple_preservation}
		 Assume that $\cC_{F,\mathrm{fl}}$ is locally finite-dimensional, i.e., every Hom space of this category is of finite dimension. Then $F'\otimes_F(-)$ respects semisimple objects and thus the finite length property.
	\end{cor}
	
	\begin{proof}
		It will suffice to prove that $F'\otimes_F M$ for an arbitrary simple object $M$ of $\cC_F$. Set $F^0=F$. We construct $F^{i}\in\Lambda_{\bar{F}/F^{i-1}}$ for $i\geq 1$ as follows: If there are a simple subobject $N\subset F^{i-1}\otimes_F M$ and an extension $F''/F^{i-1}$ in $\Lambda_{F'/F^{i-1}}$ such that $F''\otimes_{F^{i-1}} N$ is not simple then we set $F^{i}=F'$; Otherwise, we put $F^{i}=F^{i-1}$. This procedure is potentially stationary. In fact, we write $\ell_i$ for the length of $F^i\otimes_F M$. If $F^{i}\neq F^{i-1}$ then we have $\ell_{i}>\ell_{i-1}$ by Lemma \ref{lem:pseudosimple_finite} (3). On the other hand, we have
		\[\dim_F \End_{\cC_F}(M)=\dim_{F^i}\End_{\cC_{F^i}}(F^i\otimes_F M)\geq \ell_i\]
		again by Lemma \ref{lem:pseudosimple_finite} (3).
		
		We are now able to find $F''\in\Lambda_{F'/F}$ such that every simple summand of $F''\otimes_F M$ satisfies the first four equivalent conditions of Theorem \ref{thm:abs_simple}. The assertion now follows since $F''\otimes_F M$ is semisimple (Lemma \ref{lem:pseudosimple_finite} (3)). This completes the proof.
	\end{proof}

	\subsection{Relation of Loewy's classification scheme and Borel--Tits cocycles}\label{sec:relation}
	
	In this section, we discuss what the Borel--Tits cocycles represent in the context of Theorem \ref{thm:Loewy}.
	
	We start with the finite case.
	
	\begin{thm}\label{thm:End}
		Let $F'/F$ be a finite Galois extension of fields, $\cC$ be a Loewy $\Gamma_{F'/F}$-category, $M$ be a simple object of $\cC$ with $\End_{\cC}(M)=F'\id_M$, and $S$ be the simple descent system corresponding to $M$.
		\begin{enumerate}
			\item The object $M$ is self-conjugate if and only if $\End_{\DS(\cC)}(S)$ is central.
			\item If the equivalent conditions of (1) are satisfied, $(\beta^{\BT}_{M})^{-1}$ corresponds to the central division algebra $\End_{\DS(\cC)}(S)$.
		\end{enumerate}
	\end{thm}

	\begin{proof}
		Apply Lemma \ref{lem:conservative} to the forgetful functor $\pi:\DS(\cC)\to \cC$ to see that $\pi(S)$ is semisimple. In view of Theorem \ref{thm:Loewy}, there is an isomorphism $\Res_{F'/F} M\cong S^n$ for some positive integer $n$. Condition \ref{con:bcthm} implies
		\[F'\otimes_F \End_{\DS(\cC)} (S^n)\cong F'\otimes_F \End_{\DS(\cC)} (\Res_{F'/F} M)
		\cong \End_{\cC}\left(\prod_{\sigma\in\Gamma_{F'/F}} {}^\sigma M\right).\]
		This easily leads us to (1). Moreover, if the equivalent conditions of (1) are satisfied then $\End_{\DS(\cC)} (S)$ and $\End_{\DS(\cC)}(\Res_{F'/F} M)$ are central simple algebras which determine the same similarity class in the Brauer group of $F'/F$. Part (2) now follows from Proposition \ref{prop:key_computation_BT}.
	\end{proof}
	
	\begin{rem}
		See \cite[Section~4]{hayashisuper} for a theoretical proof of Theorem \ref{thm:End} (2) without computations of Proposition \ref{prop:key_computation_BT} in the setting of supercomodules, though it does not work in the present general formalism. In fact, we established a super analog of the present work for representations of affine group superschemes in \cite[Sections 3 and 4]{hayashisuper}. The arguments in \cite[Section~4]{hayashisuper} readily adapt to the non-super setting by removing the super-specific components.
	\end{rem}
	
	Theorem \ref{thm:End} (2) tells us that the division algebra of endomorphisms of the simple descent system attached to a self-conjugate simple object is determined by its Borel--Tits cocycle. We can generalize this for general simple objects by restricting the action of the Galois group.
	
	\begin{lem}\label{lem:restricting_action}
		Assume that $F'/F$ is finite. Let $F''$ be an intermediate extension of $F'/F$. Let $\cC$ be a Loewy $\Gamma_{F'/F}$-category.
		\begin{enumerate}
			\item The $F'$-linear abelian category $\cC$ is a Loewy $\Gamma_{F'/F''}$-category by restricting the action of $\Gamma_{F'/F}$ to $\Gamma_{F'/F''}$.
			\item The forgetful functor $\pi:\DS(\cC;F'/F)\to \DS(\cC;F'/F'')$ admits a right adjoint functor $\Res_{F'/F''}$. 
		\end{enumerate}
	\end{lem}
	
	\begin{proof}
		Part (1) is clear. For (2), take any descent system $(Y,\psi_\tau)\in \DS(\cC;F'/F'')$. For each $\sigma\in\Gamma_{F'/F}$, we define a morphism ${}^\sigma Y\to \prod_{\tau\in\Gamma_{F'/F''}} {}^{\sigma \tau} Y$ by
		\[{}^\sigma Y\overset{{}^\sigma \psi^{-1}_\tau}{\cong} {}^{\sigma}({}^\tau Y)
		\overset{\mu_{\sigma,\tau}}{\cong} {}^{\sigma\tau}Y.\]
		Its image is independent of the choice of $\sigma$ in the coset $\sigma\Gamma_{F'/F''}$ by the cocycle condition of $(\psi_\tau)$. The coproduct of these images for all cosets determines a subsystem of $\pi(\Res_{F'/F} M)\in\DS(\cC;F'/F'')$. Moreover, one can easily check that it exhibits the right adjoint functor.
	\end{proof}
	
	\begin{thm}\label{thm:LBT}
		Let $F'/F$ be a finite Galois extension of fields, $\cC$ be a Loewy $\Gamma_{F'/F}$-category, and $M\in\cC$ be a simple object.
		Let $S_{F(M)}\in\DS(\cC;F'/F(M))$ be the simple descent system corresponding to $M$.
		\begin{enumerate}
			\item The object $M\in\cC$ is self-conjugate with respect to the action of $\Gamma_M$.
			\item The descent system $S\coloneqq\Res_{F(M)/F} S_{F(M)}$ is simple. Moreover, $S$ corresponds to $M$.
			\item We have $\End_{\DS(\cC;F'/F(M))}(S_{F(M)})\cong \End_{\DS(\cC;F'/F)}(S)$.
		\end{enumerate}
	\end{thm}
	
	Combine this with Theorem \ref{thm:End} to conclude that if $M\in\cC$ is a simple object with $\End_{\cC}(M)=F'\id_{M}$ then $(\beta^{\BT}_M)^{-1}$ represents
	$\End_{\DS(\cC;F'/F)}(S)$, where $S\in \DS(\cC;F'/F)$ is the simple descent system corresponding to $M$.
	
	\begin{proof}
		Part (1) is clear. The assertion (3) is a straightforward from the construction of $\Res_{F(M)/F}$ and the definition of $F(M)$. Moreover, $\pi(\Res_{F(M)/F} S_{F(M)})$ is semisimple by Lemma \ref{lem:well-defined} (1). Lemma \ref{lem:conservative} then implies that $\Res_{F(M)/F} S_{F(M)}$ is semisimple. Since $\End_{\DS(\cC;F'/F(M))}(S_{F(M)})$ is a division algebra, $\Res_{F(M)/F} S_{F(M)}$ is simple (use (3)).
		
		It remains to find a nonzero morphism $\Res_{F(M)/F} S_{F(M)}\to \Res_{F'/F} M$ for (2).
		Choose a monomorphism $i:S_{F(M)}\hookrightarrow \Res_{F'/F(M)} M$.
		It is clear from the construction of $\Res_{F(M)/F}$ that $\Res_{F(M)/F}i$ is nonzero. We identify $\Res_{F(M)/F}i$ with a nonzero map to $\Res_{F'/F} M$ through the canonical isomorphism
		\[\Res_{F(M)/F}\circ \Res_{F'/F(M)}\cong \Res_{F'/F}\]
		(pass to the left adjoint functors). This completes the proof.
	\end{proof}
	
	For general $\Gamma_{F'/F}$, we consider a Borel--Tits datum which is also a Loewy datum. Henceforth let $F'/F$ be a Galois extension of fields of Galois group $\Gamma_{F'/F}$.
	
	\begin{defn}\label{defn:LBT}
		A Loewy--Borel--Tits datum consists of
		\begin{itemize}
			\item a locally small and essentially small semilinear abelian $\Gamma_{F'/F}$-category $\cC$,
			\item an abelian full subcategory $\DD(\cC)\subset\DS(\cC)$,
			\item a locally small and essentially small semilinear abelian $\Gamma_{F''/F}$-category $\cC_{F''}$
			for each $F''\in\Lambda_{F'/F}$,
			\item an $F''$-linear functor $F'\otimes_{F''}(-):\cC_{F''}\to \cC$
			for $F''\in\Lambda_{F'/F}$,
			\item an $F'''$-linear functor $F''\otimes_{F'''}(-):\cC_{F'''}\to \cC_{F''}$
			for $F'',F'''\in\Lambda_{F'/F}$ with $F'''\subset F''$,
			\item a natural transformation $\alpha^{F''}_{\sigma}:{}^\sigma(F'\otimes_{F''}(-))\cong F'\otimes_{F''} {}^{\sigma|_{F''}}(-)$ for each field $F''\in\Lambda_{F'/F}$ and $\sigma\in\Gamma_{F'/F}$,
			\item a natural transformation $\alpha^{F''/F'''}_\sigma:
			{}^\sigma (F''\otimes_{F'''}(-))\cong F''\otimes_{F'''} {}^{\sigma|_{F'''}}(-)$
			for $F'',F'''\in\Lambda_{F'/F}$ with $F'''\subset F''$ and $\sigma\in\Gamma_{F''/F}$,
			\item a natural transformation
			\[\gamma^{F^{(2)}/F^{(3)}/F^{(4)}}:
			F^{(2)}\otimes_{F^{(3)}}(F^{(3)}\otimes_{F^{(4)}}(-))\cong F^{(2)}\otimes_{F^{(4)}}(-)\]
			for $F^{(i)}\in\Lambda_{F'/F}\cup\{F'\}$ ($2\leq i\leq 4$) with $F^{(i)}\subset F^{(i-1)}$ for $i=3,4$,
		\end{itemize}
		such that
		\begin{enumerate}[label=(LBT\arabic*)]
			\item \label{con:fin_length_LBT} every object of $\cC$ is of finite length,
			\item \label{con:strict_LBT} $\DD(\cC)\subset\DS(\cC)$ is strict,
			\item \label{con:subobj_LBT} $\DD(\cC)$ is closed under formation of subobjects in $\DS(\cC)$,
			\item \label{con:exact_LBT} for $F''\in\Lambda_{F'/F}$, the functor $F'\otimes_{F''}(-)$ is exact,
			\item \label{con:bcthm_LBT} for $F''\in\Lambda_{F'/F}$ and $M,N\in\cC_{F''}$, the canonical map
			\[F'\otimes_{F''} \Hom_{\cC_{F''}}(M,N)
			\to \Hom_{\cC}(F'\otimes_{F''}M,F'\otimes_{F''}N)\]
			is an isomorphism,
			\item \label{con:associative_LBT} $\alpha^{F''}_\sigma$ and $\alpha^{F''/F'''}_\sigma$ are associative in $\sigma$,
			\item \label{con:natural_LBT} $\gamma$ are natural in $F^{(\bullet)}$,
			\item \label{con:esssurj_LBT} the functor $F'\otimes_{F''}(-):\DS(\cC_{F''})\to \DS(\cC)$ is essentially surjective onto $\DD(\cC)$, and
			\item \label{con:rationality_LBT} for every object $M\in\cC$, there exist $F''\in\Lambda_{F'/F}$ and $M_{F''}\in\cC_{F''}$ with
			\[F'\otimes_{F''} M_{F''}\cong M.\]
		\end{enumerate}
		We say a Loewy--Borel--Tits datum on a semilinear abelian $\Gamma_{F'/F}$-category $\cC$ is \emph{strong} if Condition \ref{con:Schur} is satisfied.
	\end{defn}

	\begin{ex}\label{ex:rep}
		Let $G$ be an affine group scheme over $F$.
		Then $\Rep(F''\otimes_F G)$ with $F''\in\Lambda_{F'/F}\cup\{F'\}$ and $\DD(\Rep(F'\otimes_FG))$ form a Loewy--Borel--Tits datum in the standard manner (recall Example \ref{ex:affinegroupscheme}).
	\end{ex}

	\begin{defn}[Field of rationality]\label{defn:fld_rat}
		Consider a Loewy--Borel--Tits datum as above, and $M\in\cC$ be a simple object. Then set $F(M)$ as the subfield of $F'$ fixed by all $\sigma\in\Gamma_{F'/F}$ with ${}^\sigma M\cong M$. This is finite over $F$ by \ref{con:rationality_LBT}.
	\end{defn}

	\begin{cor}\label{cor:LBT}
		Consider a Loewy--Borel--Tits datum as above. Let $M$ be a simple object of $\cC$, and $S$ be the corresponding simple descent datum.
		
		Pick $F''\in\Lambda_{F'/F}$ and $M_{F''}\in\cC_{F''}$ as in \ref{con:rationality_LBT}. Let $S_{F''}$ be the simple descent system corresponding to $M_{F''}$.
		\begin{enumerate}
			\item We have $F(M)=F(M_{F''})$.
			\item There is an isomorphism $S\cong F'\otimes_{F''} S_{F''}$.
			\item The center of $\End_{\DD(\cC)}(S)$ coincides with $F(M)$.
			\item We have $\End_{\DS(\cC_{F''})}(S_{F''})\cong\End_{\DD(\cC)}(S)$.
			\item The cocycle $(\beta^{\BT}_{M})^{-1}$ represents the central division algebra $\End_{\DD(\cC)}(S)$ over $F(M)$.
		\end{enumerate}
	\end{cor}

	\begin{proof}
		Recall that if $\sigma\in\Gamma_M$ then it restricts to an element of $\Gamma_{M_{F''}}$ (see Construction \ref{cons:BT}). Conversely, if $\sigma\in\Gamma_{F'/F}$ is a lift of an element of $\Gamma_{M_{F''}}\subset \Gamma_{F''/F}$ then $\sigma$ lies in $\Gamma_M$ (use $\alpha^{F''}_{\sigma,M_{F''}}$). These imply (1).
		
		Remark \ref{rem:Galois_descent} implies (2) and (4) (recall also Condition \ref{con:bcthm_LBT} for (2)). Part (3) is reduced to the finite case (Theorems \ref{thm:End} and \ref{thm:LBT}) by (1) and (4). Part (5) is a consequence of Theorem \ref{thm:LBT} (4).
	\end{proof}

	\begin{rem}
		We used $\gamma$ in the proof of Definition-Proposition \ref{defprop:BT} (1). In the Loewy--Borel--Tits setting above, we can verify the independence without $\gamma$ by relating the cocycle with $\End_{\DD(\cC)}(S)$.
	\end{rem}
	
	\begin{rem}
		Corollary \ref{cor:LBT} could be reduced to the self-conjugate case if we are given more data to consider a variant notion to Loewy--Borel--Tits data which is stable under restriction of the Galois group (recall Lemma \ref{lem:restricting_action} for the finite case).
	\end{rem}

	\section{Application to $(\fg,K)$-modules}\label{sec:(g,K)mod}
	
	In this section, we aim to discuss the application of the preceding arguments to $(\fg,K)$-modules and to develop further results. In this section, Harish-Chandra pairs $(\fg,K)$ will be over a field $F$ of characteristic zero with $\dim_F \fg<\infty$ and $K$ reductive unless specified otherwise.

\subsection{Structure of a Loewy--Borel--Tits datum}\label{sec:goal}

In this section, we verify that our theories on rationality are applicable to $(\fg,K)$-modules.

\begin{defn}\label{defn:HCmod}
	Let $(\fg,K)$ be a Harish-Chandra pair over a field $F$ of characteristic zero. 
	\begin{enumerate}
		\item A $\fg$-module $V$ is called \emph{finitely generated} (resp.~\emph{$Z(\fg)$-finite}) if it is so as a $U(\fg)$-module (resp.~$\Ann_{Z(\fg)}(V)$ is of finite codimension in $Z(\fg)$).
		\item A $(\fg,K)$-module is called finitely generated (resp.~$Z(\fg)$-finite) if so is it as a $\fg$-module.
		\item For a character $\chi:Z(\fg)\to F$, a $(\fg,K)$-module $X$ has infinitesimal character $\chi$ if $Z(\fg)$ acts on $X$ by $\chi$. 
		\item A $K$-module $V$ is called \emph{admissible} if $\dim_F \Hom_K(W,V)<\infty$ for every finite-dimensional representation $W$ of $K$. 
		\item We say a $(\fg,K)$-module is \emph{admissible} if so is it as a $K$-module.
	\end{enumerate}
	Let $(\fg,K)\cmod_{\mathrm{fg}}$ (resp.~$(\fg,K)\cmod_{\mathrm{fg},\adm}$) be the category of finitely generated (resp.~finitely generated and admissible) $(\fg,K)$-modules. For a homomorphism
	$\chi:Z(\fg)\to F$,
	let $(\fg,K)\cmod_{\chi}$ denote the category of $(\fg,K)$-modules with infinitesimal character $\chi$. Set
	$(\fg,K)\cmod_{\mathrm{fg},\chi}=(\fg,K)\cmod_{\mathrm{fg}}\cap (\fg,K)\cmod_{\chi}$.
\end{defn}

The following assertion is straightforward:

\begin{prop}\label{prop:subquot}
	Assume that $\dim_F\fg<\infty$.
	Then
	\[\begin{array}{cc}
		(\fg,K)\cmod_{\mathrm{fg}},&(\fg,K)\cmod_{\mathrm{fg},\adm}
	\end{array}\]
	are closed under formations of subquotient in the category $(\fg,K)\cmod$ of $(\fg,K)$-modules. In particular, they are abelian subcategories of $(\fg,K)\cmod$.
\end{prop}

We set $(\fg,K)\cmod_{\mathrm{fl},\adm}\coloneqq((\fg,K)\cmod_{\mathrm{fg},\adm})_{\mathrm{fl}}$.
In view of Proposition \ref{prop:subquot}, $(\fg,K)\cmod_{\mathrm{fl},\adm}$ consists of the admissible $(\fg,K)$-modules of finite length.

The main goal of this section is to verify:

\begin{thm}\label{thm:hc_setting}
	Let $F'/F$ be a Galois extension of fields of characteristic zero, and $(\fg,K)$ be a Harish-Chandra pair over $F$ with $\dim_F\fg<\infty$ and $K$ reductive. Then the categories
	$(F''\otimes_F\fg,F''\otimes_F K)\cmod_{\mathrm{fl},\adm}$
	with $F''\in\Lambda_{F'/F}\cup\{F'\}$ are naturally endowed with the structure of a Loewy--Borel--Tits datum. In particular, we have a bijection
	\[\Gamma_{F'/F}\backslash \Simple((F'\otimes_F\fg,F'\otimes_F K)\cmod_{\mathrm{fl},\adm})
	\cong \Simple ((\fg,K)\cmod_{\mathrm{fl},\adm}).
	\]
\end{thm}

The main difficulty for this is to show that the base change functor is well-defined. The key to its proof will be Corollary \ref{cor:semisimple_preservation}. 

\begin{lem}[{\cite[Proposition 3.6]{MR3770183}}]\label{lem:descent_Z(g)-fin}
	Let $F''/F$ be any extension of fields of characteristic zero. Let $\fg$ be a finite-dimensional Lie algebra, and $K$ be a reductive algebraic group over $F$.
	\begin{enumerate}
		\item A $U(\fg)$-module $V$ is finitely generated if and only if so is $F''\otimes_F V$ as a $U(F''\otimes_F \fg)$-module.
		\item A $K$-module $V$ is admissible if and only if so is $F''\otimes_F V$ as a $(F''\otimes_F K)$-module.
		\item A finitely generated $U(\fg)$-module $V$ is $Z(\fg)$-finite and only if $F''\otimes_F V$ is $Z(F''\otimes_F \fg)$-finite.
	\end{enumerate}
\end{lem}

Though (3) is not relevant here, let us record it for further discussions on the finiteness assumption later (Proposition \ref{prop:fl->Zf}).

\begin{proof}
	Part (1) is elementary. For (2), observe that every finite-dimensional representation $\tau'$ of $F'\otimes_F K$ is a direct summand of $F'\otimes_F \tau$ for a certain finite-dimensional representation $\tau$ of $K$. We can find such $\tau$ as follows: choose a finite generating set $S$ of $\tau'$ as a vector space over $F'$, and let $\tau$ be a $K$-submodule of the restriction of the scalar of $\tau'$ to $F$ generated by $S$. This is of finite dimension (\cite[Chapter I, 2.13]{MR2015057}, see also \cite[Section~2.1]{MR3853058} if necessary), and is a desired representation by the complete reducibility of representations of $F'\otimes_F K$. The equivalence is now an easy consequence of \cite[Chapter I, 2.10 (7)]{MR2015057}.
	
	We next prove (3). Since $\dim_F \fg<\infty$, we have a canonical isomorphism 
	\[F''\otimes_F Z(\fg)\cong Z(F''\otimes_F \fg).\]
	Let $V$ be a finitely generated $U(\fg)$-module with a finite generating subset $\{v_i\}_{i=1}^n$. Then it is evident by definition that $\Ann_{Z(\fg)}(V)=\cap_{1\leq i\leq n} \Ker(Z(\fg)\to V;~a\mapsto av_i)$. Since $F''\otimes_F(-)$ is exact, we conclude
	\[F''\otimes_F \Ann_{Z(\fg)}(V)\cong \Ann_{Z(F''\otimes_F \fg)}(F''\otimes_F V).\]
	The assertion is immediate from this isomorphism.
\end{proof}

\begin{lem}\label{lem:BT_HC}
	Let $F'/F$ be a Galois extension of fields, and $(\fg,K)$ be a Harish-Chandra pair over $F$. Then $(F''\otimes_F \fg,F''\otimes_F K)\cmod_{\mathrm{fg}}$ with $F''\in\Lambda_{F'/F}\cup\{F'\}$ are naturally endowed with the structure of a Borel--Tits datum for the base change functor of \cite[Section~3.1]{MR3853058}. A similar assertion holds true for finitely generated and admissible modules.
\end{lem}

\begin{proof}
We only prove it for finitely generated modules. The latter statement is verified in the same lines.

The base change functors respect finitely generated modules by Lemma \ref{lem:descent_Z(g)-fin}. The natural transformations $\alpha$, $\gamma$ in Definition \ref{defn:BTdatum} are given in the usual manner. We define descent data in the standard way, i.e., a descent system on a finitely generated $(F'\otimes_F \fg,F'\otimes_F K)$-module is called a descent datum if the corresponding action of $\Gamma_{F'/F}$ on the underlying vector space over $F'$ is continuous.
	We wish to prove that these give a Borel--Tits datum.
	
	Conditions \ref{con:strict_BT}, \ref{con:exact_BT}, \ref{con:associative_BT}, \ref{con:natural_BT} are clear. For \ref{con:bcthm_BT}, see \cite[Theorem 3.1.6]{MR3853058}.
	
	Condition \ref{con:esssurj_BT} is an easy consequence of the corresponding assertion for vector spaces and Lemma \ref{lem:descent_Z(g)-fin}. In fact, both categories in \ref{con:esssurj_BT} are naturally equivalent to $(\fg,K)\cmod_{\mathrm{fg}}$.
	
	It remains to prove Condition \ref{con:rationality}. Suppose that we are given a finitely generated $(F'\otimes_F\fg,F'\otimes_F K)$-module $M$. Then one can find a finite generating subset $S$ of $M$ as a $U(F'\otimes_F\fg)$-module. We regard $M$ as a $(\fg,K)$-module by restriction of scalar, and generate $S$ as a $K$-submodule to obtain a finite-dimensional generating $K$-submodule $W$ (\cite[Propositions 2.1.4 and 2.1.6]{MR3853058}).
	
	Consider the $F'\otimes_F K$-equivariant surjective map $p:F'\otimes_F (U(\fg)\otimes_F W)\to M$ defined by the action. Since the $(F'\otimes_F\fg,F'\otimes_F K)$-module $F'\otimes_F (U(\fg)\otimes_F W)$ is finitely generated, so is the kernel of $p$ (recall that $U(F'\otimes_F\fg)$ is Noetherian). Pick a finite generating subset $S'$. We consider the canonical continuous action of $\Gamma$ on $F'\otimes_F (U(\fg)\otimes_F W)$. Then the centralizer subgroup $\Gamma_{S'}\subset\Gamma$ for $S'$ is open since $S'$ is finite. Hence $M$ is defined over the subfield $(F')^{\Gamma_{S'}}\subset F'$ which is finite over $F$. Let $M_{(F')^{\Gamma_{S'}}}$ be the corresponding form.
	
	Let $F''$ be the Galois closure $F''$ of $(F')^{\Gamma_{S'}}$ in $F'$. Then $F''$ lies in $\Lambda_{F'/F}$. Moreover, we get an $(F''\otimes_F\fg,F''\otimes_F K)$-form
	$M_{F''}\coloneqq F''\otimes_{(F')^{\Gamma_{S'}}} M_{(F')^{\Gamma_{S'}}}$ of $M$.
	Lemma \ref{lem:descent_Z(g)-fin} implies that $M_{F''}$ is finitely generated.
	This completes the proof.
\end{proof}

\begin{lem}\label{lem:loc_fin_HC}
	Consider the setting of Theorem \ref{thm:hc_setting}. Then $(\fg, K)\cmod_{\mathrm{fl},\adm}$ is locally finite-dimensional.
\end{lem}

\begin{proof}
	It will suffice to prove that $\dim_F\End_{\fg,K}(V)<\infty$ for an arbitrary irreducible and admissible $(\fg,K)$-module $V$ by Proposition \ref{prop:subquot} and (b) $\Rightarrow$ (a) in Example \ref{ex:End}. Let $\tau$ be an irreducible representation of $K$ with $V(\tau)\neq 0$. Then we obtain an injective map $\End_{\fg,K}(V)\hookrightarrow \End_K(V(\tau))$ by restriction. The assertion now follows since $V$ is admissible.
\end{proof}

\begin{proof}[Proof of Theorem \ref{thm:hc_setting}]
	We wish to prove that 
	$(F''\otimes_F\fg,F''\otimes_F K)\cmod_{\mathrm{fl},\adm}$
	form
	a Loewy--Borel--Tits datum for the base change formalism of $(\fg,K)$-modules in \cite[Section~3.1]{MR3853058}, where $F''$ runs through $\Lambda_{F'/F}\cup\{F'\}$. We define a descent datum as a descent system whose corresponding action of $\Gamma_{F'/F}$ is continuous. 
	
	The base change functors respect these categories from Lemmas \ref{lem:BT_HC}, \ref{lem:loc_fin_HC}, and Corollary \ref{cor:semisimple_preservation}.
	
	Conditions \ref{con:fin_length_LBT}, \ref{con:strict_LBT}, \ref{con:subobj_LBT}, \ref{con:exact_LBT}, \ref{con:associative_LBT} are clear. For \ref{con:bcthm_LBT}, see \cite[Theorem 3.1.6]{MR3853058}. Conditions \ref{con:esssurj_LBT} and \ref{con:rationality_LBT} are immediate from Lemmas \ref{lem:BT_HC} and \ref{lem:descent_Z(g)-fin}. This completes the proof.
\end{proof}

\begin{rem}\label{rem:assumption}
	It is evident by the proof that the characteristic of $F$ can be positive in Lemma \ref{thm:hc_setting} (1) and (3). Similarly, the assertion of Lemma \ref{lem:BT_HC} for finitely generated modules keeps true without the assumptions on $F$ and $K$.
\end{rem}

\subsection{Absolute irreducibility}\label{sec:abs_irr}

In Theorem \ref{thm:abs_simple}, we discussed generalities on the pseudo-absolute simplicity (irreducibility). Here we aim to verify a stronger result in the setting of $(\fg,K)$-modules:

\begin{thm}\label{thm:abs_irr_hc}
	Let $(\fg,K)$ be a Harish-Chandra pair over a field $F$ of characteristic zero with $\dim_F \fg<\infty$ and $K$ reductive. Then for an irreducible admissible $(\fg,K)$-module $V$, the following conditions are equivalent:
	\begin{enumerate}
		\renewcommand{\labelenumi}{(\alph{enumi})}
		\item $\bar{F}\otimes_F V$ is irreducible;
		\item $F'\otimes_F V$ is irreducible for every finite extension $F'/F$ of fields;
		\item $F'\otimes_F V$ is irreducible for every extension $F'/F$ of fields;
		\item $\End_{\fg,K}(V)=F\id_V$.
	\end{enumerate}
	If these equivalent conditions are satisfied then $V$ is called absolutely irreducible (\cite[Definition 3.3]{MR3770183}).
\end{thm}

\begin{lem}\label{lem:bc_bij}
	Let $K$ be an affine group scheme over an algebraically closed field $F$, and $F'$ be any field extension of $F$. Then the base change functor determines a bijection
	\begin{equation}
		\Simple(\Rep K)\cong\Simple(\Rep (F'\otimes_F K))\label{eq:bc}
	\end{equation}
\end{lem}

In fact, we see from the argument below that a similar assertion holds true for right comodules over coalgebras.

\begin{proof}
	We apply Jacobson's density theorem to the dual algebra $F[K]^\vee$ of the coordinate ring of $K$ to see that $F'\otimes_F(-)$ respects irreducible representations (\cite[PROPOSITION 1.1.1, SECTION 2.1]{MR252485}). In particular, we obtain the injective map \eqref{eq:bc} from left to right.
	
	To see that this map is surjective, let $U$ be any irreducible representation of $F'\otimes_F K$. Choose an irreducible subrepresentation $V$ of $\Res_{F'/F} U$. Then we have $\Hom_{F'\otimes_F K}(F'\otimes_F V,U)\neq 0$. Since $F'\otimes_FV$ and $U$ are irreducible, they are isomorphic to each other.
\end{proof}

\begin{lem}\label{lem:K-type}
	Let $(\fg,K)$ be a Harish-Chandra pair over an algebraically closed field of characteristic zero with $\dim_F \fg<\infty$ and $K$ reductive. Then for an irreducible admissible $(\fg,K)$-module $V$ and an irreducible representation $\tau$ of $K$, $\Hom_K(\tau,V)$ is an irreducible $U(\fg)^K$-module for the action on the target.
\end{lem}

\begin{proof}
	This is verified along the same lines as \cite[3.5.4. Proposition, 3.9.7, 3.9.8]{MR929683}. Here we only outline how to define the Hecke algebra $H$ in an algebraic way.
	
	To save space, set $S\coloneqq \Simple(\Rep K)$. We will identify $\tau\in S$ with an irreducible representation of $K$. For $\tau\in S$, we put the canonical structure of a coalgebra on $\End_F(\tau)$. We also regard it as a representation of $K$ for the action on the target. Then the algebraic Peter--Weyl theorem asserts an isomorphism
	\[\oplus_{\tau\in S}\End_F(\tau)\cong F[K]\]
	as coalgebras and representations of $K$, where $F[K]$ is regarded as the right regular representation of $K$. If we identify $\End_F(\tau)^c=\End_F(\tau)^\vee$ with $\End_F(\tau)$ via the trace pairing, we have a commutative diagram
	\[\begin{tikzcd}
		\prod_{\tau\in S} \End_F(\tau)\ar[r, "\sim"]&F[K]^\vee\\
		\oplus_{\tau\in S} \End_F(\tau)\ar[r, "\sim"]\ar[u, hook]&F[K]^c\ar[u,hook]
	\end{tikzcd}\]
	of canonical maps. Recall that $F[K]^\vee$ is naturally endowed with the structure of an algebra (\cite[Proposition 1.1.1]{MR252485}). Correspondingly, $\prod_{\tau\in S} \End_F(\tau)$ is the product of the algebras $\End_F(\tau)$. By restriction, we obtain the structures of approximately unital algebras in the sense of \cite[DEFINITION A.3]{MR1330919} on the bottom terms.
	
	Write the evaluation map $F[K]^c\otimes_F F[K]\to F$ (i.e., the counit) as
	$T\otimes a\mapsto T(a)$.
	Let
	\begin{equation}
		F[K]\otimes_F F[K]^c\to F[K]^c;~a\otimes T\mapsto aT\label{eq:scalar_multiplication}
	\end{equation}
	be the map given by
	\[F[K]\otimes_F F[K]^c\otimes_FF[K]\to F;~a\otimes T\otimes b\mapsto T(ab).\]
	
	We define a map
	\begin{equation}
		F[K]^c\otimes_F U(\fg)\cong U(\fg)\otimes_F F[K]\label{eq:symmetry}
	\end{equation}
	by
	\[\begin{split}
		F[K]^c\otimes_F U(\fg)
		&\to F[K]^c\otimes_F U(\fg)\otimes_F F[K]\\
		&\cong F[K]\otimes_F F[K]^c\otimes_F U(\fg)\\
		&\to F[K]^c \otimes_F U(\fg).
	\end{split}
	\]
	The first map is defined by the coaction on $U(\fg)$. The second map is given by switching the factors as $T\otimes u\otimes a\mapsto a\otimes T\otimes u$. The last map is given by \eqref{eq:scalar_multiplication}. Its inverse is constructed in a similar way.
	
	We now define a product on $U(\fg)\otimes_F F[K]^c$ by
	\[\begin{split}
		U(\fg)\otimes_F F[K]^c
		\otimes_F U(\fg)\otimes_F F[K]^c
		&\cong U(\fg)\otimes_F U(\fg)\otimes_F F[K]^c\otimes_F F[K]^c\\
		&\to U(\fg)\otimes_F F[K]^c,
	\end{split}
	\]
	where the first isomorphism is obtained by applying \eqref{eq:symmetry} to the middle two factors, and the last map is given by the multiplication. One can prove that $U(\fg)\otimes_F F[K]^c$ is an (approximately unital) algebra for this map.
	
	Consider the differential action of $U(\fk)$ on $F[K]^c$. Then the product map in the previous paragraph descends to that on $U(\fg)\otimes_{U(\fk)} F[K]^c$. The vector space $H\coloneqq U(\fg)\otimes_{U(\fk)} F[K]^c$ is an (approximately unital) algebra for this.
\end{proof}

\begin{proof}[Proof of Theorem \ref{thm:abs_irr_hc}]
	The equivalences (a) $\iff$ (b) $\iff$ (d) follow from Lemma \ref{lem:BT_HC} and Theorem \ref{thm:abs_simple}. It is clear that (c) implies (b). It remains to prove (a) $\Rightarrow$ (c). We may replace $F$ with the algebraic closure $\bar{F}$ to assume that $F$ is algebraically closed. Replacing $F'$ in (c) with $\bar{F}'$, we may prove that $F'\otimes_F V$ is irreducible if $F'/F$ is an extension of algebraically closed fields.
	
	Set $S\coloneqq\Simple(\Rep K)$.
	Let $\tau\in S$. Then $V(\tau)$ is nonzero if and only if the subspace $(F'\otimes_F V)(F'\otimes_F \tau)$ is nonzero. Moreover, if $V(\tau)\neq 0$ then the left $U(F'\otimes_F \fg)^{F'\otimes_F K}$-module $\Hom_{F'\otimes_F K}(F'\otimes_F \tau,F'\otimes_F V)$ is simple by Lemma \ref{lem:K-type} and Jacobson's density theorem.
	
	Let $W$ be a nonzero $(F'\otimes_F \fg,F'\otimes_F K)$-submodule of $F'\otimes_F V$. In view of Lemma \ref{lem:bc_bij}, one can find $\tau\in S$ with $W(F'\otimes_F \tau)\neq 0$. We fix this $\tau$. Then the previous paragraph implies
	$W(F'\otimes_F \tau)=(F'\otimes_F V)(F'\otimes_F \tau)$.
	Since $V$ is irreducible, $V$ is generated by $V(\tau)$. Hence $F'\otimes_F V$ is generated by
	\[F'\otimes_F V(\tau)\cong (F'\otimes_F V)(F'\otimes_F \tau)=W(F'\otimes_F \tau).\]
	This shows $F'\otimes_F V=W$. Hence we conclude that $F'\otimes_F V$ is irreducible. This completes the proof.
\end{proof}

\begin{cor}\label{cor:rationality_Z(g)-fin,finlength}
	Consider the setting above Lemma \ref{lem:fg/Z(g)}. Let $V$ be an irreducible admissible $(\fg,K)$-module. Then there exists a finite extension $F'/F$ in $\bar{F}$ such that $F'\otimes_F V$ is decomposed into finitely many absolutely irreducible $(F'\otimes_F \fg,F'\otimes_F K)$-submodules. In particular, $F''\otimes_F V$ is semisimple as an $(F''\otimes_F\fg,F''\otimes_FK)$-module for any field extension $F''/F$.
\end{cor}

\begin{proof}
This is immediate from the proof of Corollary \ref{cor:semisimple_preservation} (see also Theorem \ref{thm:abs_irr_hc}).
\end{proof}

\begin{cor}\label{cor:descent_flZf}
	Consider the setting above Lemma \ref{lem:fg/Z(g)}. Then for a $(\fg,K)$-module $V$ and a field extension $F'/F$, $V$ is $Z(\fg)$-finite of finite length if and only if $F'\otimes_F V$ is $Z(F'\otimes_F\fg)$-finite of finite length.
\end{cor}

\begin{proof}
	It is easy to show that a $(\fg,K)$-module of finite length is finitely generated (see Proposition \ref{prop:fl->Zf}). Since $F'\otimes_F(-)$ is exact and conservative, it reflects the finite length property. The assertion now follows from Lemma \ref{lem:descent_Z(g)-fin} and Corollary \ref{cor:rationality_Z(g)-fin,finlength}.
\end{proof}

\subsection{More on Borel--Tits cocycles}\label{sec:more}

Recall that categories of finitely generated $(\fg,K)$-modules form a Borel--Tits datum (Lemma \ref{lem:BT_HC}). If $F$ is algebraically closed, irreducible $(\fg,K)$-modules only have scalar endomorphisms (\cite[Proposition 5.4.6 (1)]{hayashijanuszewski}). We thus obtain the notion of Borel--Tits cocycles for irreducible $(\bar{F}\otimes_F\fg,\bar{F}\otimes_FK)$-modules.

In this section, we aim to generalize \cite[Theorem 5.5 and Proposition 5.8]{MR3770183}. We also use our theories of rationality to give a minimal field of definition of an irreducible module.

Throughout this small section, let $(\fg,K)$ be a Harish-Chandra pair over a field $F$ of characteristic zero with $\dim_F \fg<\infty$ and $K$ reductive.

\subsubsection{Local-global principle}

Suppose that $F$ is a number field. Let $V$ be a self-conjugate irreducible admissible $(\bar{F}\otimes_F \fg,\bar{F}\otimes_F K)$-module. For each place $v$ of $F$, we fix an embedding $\bar{F}\hookrightarrow \bar{F}_v$.

\begin{prop}\label{prop:lg}
	\begin{enumerate}
		\item For each place $v$, $\bar{F}_v\otimes_{\bar{F}} V$ is irreducible and self-conjugate. Moreover, the image of $\beta^{\BT}_V\in H^2(\Gamma_F,\bar{F}^\times)$ in $H^2(\Gamma_{F_v},\bar{F}^\times_v)$ agrees with $\beta^{\BT}_{\bar{F}_v\otimes_{\bar{F}} V}$.
		\item The cocycle $\beta^{\BT}_V$ is trivial if and only if so are $\beta^{\BT}_{\bar{F}_v\otimes_{\bar{F}} V}$ for all places $v$ of $F$.
	\end{enumerate}
\end{prop}

\begin{proof}
	The irreducibility follows from Theorem \ref{thm:abs_irr_hc}. The self-conjugacy is evident. Therefore $\beta^{\BT}_{\bar{F}_v\otimes_{\bar{F}} V}$ is defined. The last statement of (1) follows from the numerical description of the Borel--Tits cocycle (Construction \ref{cons:BT} and Proposition \ref{prop:key_computation_BT}). Part (2) is a consequence of the Albert--Brauer--Hasse--Noether theorem (\cite[Theorem 14.11]{MR4174395}).
\end{proof}

\begin{rem}
	We can remove the self-conjugacy assumption on $V$ by replacing $F$ with $F(V)$.
\end{rem}

\subsubsection{Computation by $K$-type}

In this section, we give a straightforward generalization of \cite[Proposition 5.8]{MR3770183}:

\begin{prop}\label{prop:mult_one}
	Let $V$ be an irreducible $(\bar{F}\otimes_F\fg,\bar{F}\otimes_FK)$-module, and $\tau$ be a self-conjugate $\bar{F}\otimes_FK$-type of multiplicity one in $V$. Then we have $\beta^{\BT}_{V}=\beta^{\BT}{\tau}$.
\end{prop}

\begin{proof}
	Let $\sigma\in\Gamma_F$. Then any isomorphism ${}^\sigma V\cong V$ restricts to ${}^\sigma \tau \cong \tau$. The equality now follows from the construction of the Borel--Tits cocycle.
\end{proof}

\subsubsection{Minimal field of definition}\label{sec:min}

Roughly speaking, irreducible Harish-Chandra modules over $F$ are obtained by Galois descent of a sum of an irreducible Harish-Chandra modules over $\bar{F}$ and its Galois twists. In this section, we aim to descend a single irreducible Harish-Chandra module over $\bar{F}$ over a certain finite extension of $F$.

\begin{prop}\label{prop:min_field}
	Let $V$ be an irreducible $(\bar{F}\otimes_F \fg,\bar{F}\otimes_F K)$-module.
	\begin{enumerate}
		\item If $V$ admits an $F'$-form for a finite extension $F'/F$ in $\bar{F}$ then we have $F(V)\subset F'$.
		\item Let $F'/F(V)$ be a finite extension in $\bar{F}$ such that the image of $\beta^{\BT}_V$ in $H^2(\Gamma_{\bar{F}/F'},\bar{F}^\times)$ is trivial. Then $V$ admits an $F'$-form.
	\end{enumerate}
	
\end{prop}

\begin{proof}
	Part (1) is clear by definition of $F(M)$. 
	
	Consider the field $F'$ in (2). Since $\Gamma_{\bar{F}/F'}\subset \Gamma_{\bar{F}/F(M)}$, $V$ is self-conjugate with respect to the twists by $\Gamma_{\bar{F}/F'}$. Part (2) is now deduced by applying Definition-Proposition \ref{defprop:BT} to 
	\[(F''\otimes_{F'}(F'\otimes_F\fg),(F''\otimes_{F'}(F'\otimes_F K))\cmod_{\mathrm{fg}}\]
	with $F''\in\Lambda_{\bar{F}/F'}\cup\{\bar{F}\}$.
\end{proof}

\begin{rem}\label{rem:existence}
	The field $F'$ in Proposition \ref{prop:min_field} always exists. In fact, recall that we have a canonical isomorphism
	\[\varinjlim_{F'\in \Lambda_{\bar{F}/F(V)}} H^2(\Gamma_{\bar{F}/F'},(F')^\times)
	\cong H^2(\Gamma_{\bar{F}/F(V)},\bar{F}^\times)\]
	(see \cite[Proposition 4.18]{MR4174395} for example). One can find $F'\in \Lambda_{\bar{F}/F(V)}$ such that $\beta^{\BT}_V\in H^2(\Gamma_{\bar{F}/F(V)},\bar{F}^\times)$ lies in (the image of) $H^2(\Gamma_{\bar{F}/F'},(F')^\times)$. It is evident by definition of the maps between the Galois cohomology groups that this $F'$ satisfies the condition.	
\end{rem}

\begin{rem}
	It is evident by the arguments above and Remark \ref{rem:assumption} that $K$ is not necessarily reductive. Similarly, the characteristic of $F$ can be positive if we replace $\bar{F}$ with $F^{\sep}$.
\end{rem}

\subsection{Remarks on finiteness conditions}\label{sec:rem}
	
	In this section, we review purely algebraic proofs of Harish-Chandra's finiteness theorem for relations and requirements of finiteness conditions (finite length, finite generation, admissibility, $Z(\fg)$-finiteness) in Sections~\ref{sec:goal}-\ref{sec:more}, based on \cite[Sections 3.1 and 3.2]{MR3770183}.

	\begin{prop}\label{prop:fl->Zf}
		Let $(\fg,K)$ be a Harish-Chandra pair. 
		\begin{enumerate}
			\item A $(\fg,K)$-module $V$ of finite length is finitely generated, Moreover, $V$ is $Z(\fg)$-finite if $\fg$ is reductive.
			\item A finitely generated and admissible $(\fg,K)$-module is $Z(\fg)$-finite.
			\item Suppose that $\fg$ is reductive and that $(\fg,\fk)$ is symmetric. Then a finitely generated and $Z(\fg)$-finite $(\fg,K)$-module is admissible.
		\end{enumerate}
	\end{prop}
	
	\begin{lem}\label{lem:adm}
		Let $K$ be a reductive algebraic group over a field of characteristic zero. Then a $K$-module is admissible if and only if so is it as a $K^0$-module.
	\end{lem}
	
	\begin{proof}
		This is verified along the same lines as \cite[Proof of Proposition 7.221]{MR1330919}.
	\end{proof}

	\begin{lem}\label{lem:fg/Z(g)}
		Let $(\fg,K)$ be a Harish-Chandra pair with $\fg$ reductive, $K$ connected split reductive, and $(\fg,\fk)$ symmetric. Let $V$ be a finitely generated $(\fg,K)$-module, and $\tau$ be an irreducible representation of $K$. Then $V(\tau)$ is a finitely generated $Z(\fg)$-submodule of $V$. 
	\end{lem}
	
	\begin{proof}
		Since $\tau$ is absolutely irreducible, we may assume that $F$ is algebraically closed. Then we can verify the assertion along almost the same lines as \cite[3.4.1. Theorem]{MR929683}. Here we explain how to modify them for a purely algebraic proof.
		
		\cite[3.4.3. Lemma]{MR929683} is easily verified by using the fact that $(-)^K$ is exact. For its details, we refer to \cite[Lemmas 10.13, 10.15, 10.18]{MR1330919}. To prove them purely algebraically, suppose that $\fk\subset\fg$ is the fixed point Lie subalgebra for an involution $\theta$ of $\fg$. 
		
		Set $\fp=\fg^{-\theta}$.
		Take a connected and simply connected semisimple algebraic group $\tilde{G}^{\mathrm{ss}}$ whose Lie algebra is the derived subalgebra of $\fg$. Consider tori $T^+$ and $T^-$ whose Lie algebras are $Z_\fg\cap \fk$ and $Z_{\fg}\cap\fp$ respectively. Set $\tilde{G}=T^+\times T^-\times \tilde{G}^{\mathrm{ss}}$. Its Lie algebra is naturally identified with $\fg$. Moreover, $\theta$ lifts to an involution of $\tilde{G}$ which we denote by the same symbol.
		
		Fix a maximal $\theta$-split torus $\tilde{A}\subset \tilde{G}$ in the sense of \cite[2.1]{MR1066573} and a $\theta$-stable maximal torus $\tilde{H}\subset \tilde{G}$ containing $\tilde{A}$. Let $\fa$ and $\fh$ denote the Lie algebra of $\tilde{A}$ and $\tilde{H}$ respectively. Then we find $\fa=\fh^{-\theta}$ (consider the involution $h\mapsto \theta(h)^{-1}$ on $\tilde{H}$).
		
		For \cite[Lemma 10.15]{MR1330919}, we observe that $\tilde{A}$ contains a maximal $\theta$-split torus of $\tilde{G}^{\mathrm{ss}}$ since the center of $\tilde{G}$ is diagonalizable. Therefore the line below (10.17) remains valid in the present algebraic setting because of a similar argument to the first paragraph of \cite[Proof of Corollary 6.53]{MR1920389} (cf.~\cite[9.5 Remark]{MR1066573}).

		One can prove \cite[Lemma 10.13]{MR1330919} without the real structure by the injectivity of Chevalley's restriction theorem which has purely algebraic proofs (\cite[Chapter VIII, \S 8, THEOREM 1. (i)]{MR2109105}). In fact, let $W$ be the Weyl group of $(\fg,\fh)$. We regard $\fp$ as representations of $K$ for the adjoint action.
		Consider the inclusion maps $\fa\hookrightarrow \fh\hookrightarrow\fg$ and the projections $\pr_{\fg/\fp}:\fg\to \fp$, $\pr_{\fh/\fa}:\fh\to\fa$ (recall from the fourth paragraph that $\fa=\fh^{-\theta}$). Their pullbacks give rise to a commutative diagram
		\[\begin{tikzcd}
			F[\fp]^K\ar[r, hook,"\pr^\ast_{\fg/\fp}"]\ar[rrrd]&F[\fg]^{\fg}\ar[r, "\sim"', "(-)|_{\fh}"]
			&F[\fh]^W\ar[r,hook]&F[\fh]\ar[r]&F[\fa]\\
			&&&F[\fa]\ar[ru,equal]\ar[u, hook, "\pr^\ast_{\fh/\fa}"]
		\end{tikzcd}\]
		The vertical and left horizontal maps are injective since $\pr_{\fh/\fa}$ and $\pr_{\fg/\fp}$ are surjective. The second horizontal arrow is an isomorphism (injective) from Chevalley's restriction theorem. The third horizontal map is just an inclusion. We conclude that the composite horizontal map $F[\fp]^K\to F[\fa]$ is injective.
		
		\cite[Proof of Lemma 10.18]{MR1330919} literally keeps true.
	\end{proof}
	
	\begin{proof}[Proof of Proposition \ref{prop:fl->Zf}]
		Part (2) follows from a similar argument to \cite[Chapter VII, \S2, EXAMPLES. 2)]{MR1330919}. 
		
		To prove (1), suppose that we are given a $(\fg,K)$-module $V$ of finite length. Then $V$ is finitely generated by \cite[Proposition 5.4.6 (2)]{hayashijanuszewski} (reduce to the case that $V$ is irreducible).
		
		Assume that $\fg$ is reductive. The same argument as \cite[Corollary 3.2]{MR3770183} implies that $V$ is $Z(\fg)$-finite by virtue of \cite[Proposition 5.4.6 (1)]{hayashijanuszewski}.
		
		Finally, we prove (3). In view of Lemmas \ref{lem:descent_Z(g)-fin}, \ref{lem:adm}, and \cite[Proposition 3.6]{MR3770183}, we may assume $F=\bar{F}$ and $K=K^0$. In this case, the assertion is immediate from Lemma \ref{lem:fg/Z(g)}.
	\end{proof}

	\begin{thm}[Kitagawa]\label{thm:fin_length}
		Let $(\fg,K)$ be a Harish-Chandra pair over a field $F$ of characteristic zero with $\fg$, $K$ reductive and $(\fg,\fk)$ symmetric. Then a $(\fg,K)$-module is of finite length if and only if it is finitely generated and $Z(\fg)$-finite.
	\end{thm}

	This is stated in \cite[Theorem 3.7]{MR3770183}. In Januszewski's proof, a field embedding of $F$ into $\mathbb{C}$ is required to apply Harish-Chandra's analytic argument. Here, we provide a purely algebraic proof based on the localization theory developed by Beilinson--Bernstein. For this, let us recall basic notions and well-known facts:
	
	\begin{defn}[{\cite[5.2]{MR1021510}}, {\cite[Chapter VII, Section~9]{MR1330919}}]\label{defn:intantidominant}
		Let $\fg$ be a reductive Lie algebra over an algebraically closed field, and $\fh$ be a Cartan subalgebra of $\fg$. Fix a positive system of $(\fg,\fh)$. A character $\lambda\in\fh^\vee$ is \emph{integrally anti-dominant} if $\langle\alpha^\vee,\lambda\rangle$ is not a positive integer for any positive coroot $\alpha^\vee$ (see \cite[Chapter VIII, \S2.2 (2)]{MR2109105} for the definition of coroots).
	\end{defn}
	
	\begin{lem}\label{lem:Weylgrouptranslation}
		Consider the setting of Definition \ref{defn:intantidominant}. Let $W$ be the Weyl group of $(\fg,\fh)$. Then for every character $\lambda\in\fh^\vee$, there exists an element $w\in W$ such that $w\lambda$ is integrally anti-dominant.
	\end{lem}
	
	\begin{proof}
		Let $\Delta$ be the set of roots of $(\fg,\fh)$. Set $\Delta(\lambda)\coloneqq \{\alpha\in\Delta:~\langle \alpha^\vee,\lambda\rangle \in\bZ\}$. This is a root subsystem of $\Delta$, i.e., it is stable under the simple reflections of $\Delta(\lambda)$.
		
		Notice that $\Delta^+\cap\Delta(\lambda)$ is a positive system of $\Delta(\lambda)$. Let $\rho_\lambda$ be the half sum of positive roots in $\Delta(\lambda)$. Choose a sufficiently large integer $n$, so that $-\lambda+\frac{1}{n}\rho_\lambda$ is $\Delta(\lambda)$-regular. Let $\Delta(\lambda)^+$ be the attached positive system.
		
		Let $w$ be the element of the Weyl group of $\Delta(\lambda)$ such that $w\Delta(\lambda)^+=\Delta^+\cap\Delta(\lambda)$. This satisfies the condition in the statement.
	\end{proof}

\begin{cons}\label{cons:globalization}
	Let $(\fg,K)$ be a Harish-Chandra pair over an algebraically closed field $F$ of characteristic zero with $\fg$ reductive and $K$ connected reductive. 
	
	Let $\cB$ be the flag variety of $\fg$. Fix a Cartan subalgebra $\fh\subset \fg$ and a positive system of $(\fg,\fh)$. Pick $\lambda\in\fh^\vee$. Let $\cD_\lambda$ be as in \cite[6.1]{MR1021510}. To be precise, we apply the construction of \cite[6.1]{MR1021510} to $\lambda|_{\fh\cap\fg^{\mathrm{ss}}}$. For the action of $Z_{\fg}$, we set a map from $Z_\fg$ to the algebra of global sections of $\cD_\lambda$ by $z\mapsto \chi_\lambda(z)$. Note also that $\cD_\lambda$ is naturally endowed with a $K$-equivariant structure.
	
	Let $\Coh(\cD_\lambda,K)$ be the category of $K$-equivariant coherent left $\cD_\lambda$-modules. Then taking the global sections determines a functor
	\begin{equation}
		\Gamma(\cB,-):\Coh(\cD_\lambda,K)\to (\fg,K)\cmod_{\chi_\lambda}\label{eq:global_section}
	\end{equation}
	by \cite[Proposition 6.2.3]{MR1021510}.

Following \cite[Theorem 6.4.2]{MR1021510}, we define $\Coh^e(\cD_\lambda,K)\subset \Coh(\cD_\lambda,K)$ as the full subcategory consisting of objects $\cM\in \Coh(\cD_\lambda,K)$ with the following properties:
\begin{enumerate}
	\item[(i)] the counit $\cD_\lambda\otimes_{U(\fg)}\Gamma(\cB,\cM)\to \cM$ is surjective, and
	\item[(ii)] for any nonzero subobject $\cN\subset \cM$ in $\Coh(\cD_\lambda,K)$, we have $\Gamma(\cB,\cN) \neq 0$.
\end{enumerate}
Here the notation $\Coh^e(\cD_\lambda,K)$ is taken from \cite[Section~11.2]{MR2357361}.
A similar argument to \cite[Theorem 6.4.2]{MR1021510} implies a categorical equivalence
\begin{equation}
	\Gamma(\cB,-):\Coh^e(\cD_\lambda,K)\simeq (\fg,K)\cmod_{\mathrm{fg},\chi_\lambda}\label{eq:cateq}
\end{equation}
if $\lambda$ is integrally anti-dominant.
\end{cons}

	\begin{proof}[Proof of Theorem \ref{thm:fin_length}]
		
		We verified the ``only if'' direction. To prove the ``if'' direction, we may assume $F$ algebraically closed by taking base change to an algebraic closure $\bar{F}$ of $F$. We may and do assume that $K=K^0$.
				
		Let $X$ be a finitely generated and $Z(\fg)$-finite $(\fg,K)$-module. Taking the primary decomposition of $X$ (\cite[Proposition 7.20]{MR1330919}), we may assume that $X$ has generalized infinitesimal character $\chi$. We define an increasing sequence of $(\fg,K)$-submodules of $V$ as follows:
		\begin{enumerate}
			\item Set $X_0=0$.
			\item For $i\geq 1$, set $\bar{X}_i=\cap_{z\in Z(\fg)} \Ann_{X/X_{i-1}}(z-\chi(z))$. Let $X_i$ be the preimage of $\bar{X}_i$ in $X$.
		\end{enumerate}
		Since $X$ has generalized infinitesimal character $\chi$, we have $X=X_n$ for sufficiently large positive integer $n$. Moreover, it follows from the construction of $X_\bullet$ that each successive quotient $X_i/X_{i-1}$ has infinitesimal character $\chi$. Therefore we may assume that $X$ has infinitesimal character $\chi$.
		
		Pick a Cartan subalgebra of $\fg$ and a positive system of $(\fg,\fh)$. Then we can write $\chi=\chi_\lambda$ for some $\lambda\in\fh^\vee$. In view of Lemma \ref{lem:Weylgrouptranslation}, we may and do assume that $\lambda$ is integrally anti-dominant. 
		
		Let $\cB$, $\cD_\lambda$, $\Coh(\cD_\lambda,K)$, $\Gamma(\cB,-)$, and $\Coh^e(\cD_\lambda,K)$ be as in Construction \ref{cons:globalization}.
		Let $X$ be as before. Then $X$ is realized as $\Gamma(\cB,\mathcal{X})$ for some $\mathcal{X}\in \Coh^e(\cD_\lambda,K)$ by \eqref{eq:cateq}. Regard $\mathcal{X}$ as a $K$-equivariant coherent $\cD_\lambda$-module. Since the number of $K$-orbits in $\cB$ is finite, $\mathcal{X}$ is holonomic. In particular, $\mathcal{X}$ is of finite length in $\Coh(\cD_\lambda,K)$. Since the global section functor \eqref{eq:global_section} is exact (\cite[Theorem 6.3.1]{MR1021510}), it will suffice to verify that for a simple $K$-equivariant coherent $\cD_\lambda$-module $\cM$, $\Gamma(\cB,\cM)$ is irreducible or zero as a $(\fg,K)$-module.
		
		Write $M=\Gamma(\cB,\cM)$. Let $N$ be a nonzero $(\fg,K)$-submodule of $M$. Passing to the categorical equivalence \eqref{eq:cateq}, we obtain a monomorphism $i:\cN\to \cM$ in $\Coh^e(\cD_\lambda,K)$. Since $\cM$ is simple, $i$ is an epimorphism in $\Coh(\cD_\lambda,K)$ and thus in $\Coh^e(\cD_\lambda,K)$. Since $\Coh^e(\cD_\lambda,K)$ is an abelian category by \eqref{eq:cateq}, $i$ is an isomorphism. This shows that $M=N$. This completes the proof.
		
	\end{proof}
	
	Therefore we can remove `adm' in Theorem \ref{thm:hc_setting}, Lemma \ref{lem:loc_fin_HC} in the setting of Proposition \ref{prop:fl->Zf} (3). In Theorem \ref{thm:abs_irr_hc}, Lemma \ref{lem:K-type}, Corollary \ref{cor:rationality_Z(g)-fin,finlength}, and Proposition \ref{prop:lg}, the admissibility follows from the irreducibility. Some of the refined statements are collected in Section~\ref{sec:results} for convenience.

	\subsection{Application to cohomological irreducible essentially unitarizable representations}\label{sec:coh}

In this section, we discuss consequences of the Vogan--Zuckerman theory (\cite{MR762307}, see also \cite[Sections 5.2 and 5.4]{hayashijanuszewski} for a setting of disconnected reductive Lie groups).

Let $F$ be a subfield of $\bR$. We take the algebraic closure $\bar{F}$ of $F$ in $\bC\supset\bR$ to regard $\bar{F}\subset\bC$. Let $G$ be a connected reductive algebraic group over $F$ with an involution $\theta$, and $K$ be an open subgroup of $G^\theta$. Assume that $\bR\otimes_F \theta$ is Cartan. For discussions below, let us choose a maximal torus $T$ of $K$. Let $H$ be the attached fundamental Cartan subgroup, i.e., the centralizer of $T$ in $G$.

Let $V$ be an irreducible representation of $\bC\otimes_F G$. Let $X$ be a cohomological irreducible essentially unitarizable $(\bC\otimes_F \fg,\bC\otimes_F K)$-module with coefficient $V$. We discuss its rationality.

According to \cite[Theorem 5.4.14]{hayashijanuszewski}, $X$ is a cohomologically induced module. That is, we can find a $\bC\otimes_F\theta$-stable parabolic subgroup $Q$ containing $\bC\otimes_F H$ such that $X$ is the Zuckerman derived functor module attached to a one dimensional representation $\pi$. In fact, let $L$ be the Levi subgroup of $Q$ containing $\bC\otimes_F H$, and $U^-$ be the opposite unipotent radical of $Q$. Then $\pi$ is a one-dimensional $(\fl,L\cap (\bC\otimes_F K))$-module whose underlying $\fl$-module coincides with the dual of the $U^-$-invariant part of $V$. Observe that $V$, $Q$, $L$, and $\pi$ admit $\bar{F}$-forms $V_{\bar{F}}$, $Q_{\bar{F}}$, $L_{\bar{F}}$, and $\pi_{\bar{F}}$ respectively (Lemma \ref{lem:bc_bij}). Take the cohomological induction over $\bar{F}$ to obtain an $\bar{F}$-form $X_{\bar{F}}$ of $X$ (\cite[Theorems D and H]{MR3853058}).

\begin{prop}\label{prop:form}
	There exists a unique $\bar{F}$-form of $X$ up to isomorphism.
\end{prop}

\begin{proof}
	We only see the uniqueness. It follows from Schur's lemma and \cite[Theorem 3.1.6]{MR3853058}.
\end{proof}

We next discuss the self-conjugacy of $X_{\bar{F}}$. Let $\sigma\in\Gamma_F$. 

\begin{cons}[Tits' $\ast$-action]
	Suppose that ${}^\sigma Q_{\bar{F}}$ is $K(\bar{F})$-conjugate to $Q_{\bar{F}}$. Then we can find $k_\sigma\in K(\bar{F})$ such that the equality ${}^\sigma Q_{\bar{F}}=k_\sigma Q_{\bar{F}}k^{-1}_\sigma$ holds. We may and do assume that $k_\sigma$ normalizes $\bar{F}\otimes_FT$. Let $\sigma\ast\pi$ be the $(\fl_{\bar{F}},L_{\bar{F}}\cap (\bar{F}\otimes_F K))$-module obtained by the $k_\sigma$-twist of $\pi_{\bar{F}}$.
\end{cons}

\begin{lem}
The isomorphism class of the $(\fl_{\bar{F}},L_{\bar{F}}\cap (\bar{F}\otimes_F K))$-module
$\sigma\ast\pi_{\bar{F}}$
is independent of choice of $k_\sigma$ above.
\end{lem}

\begin{proof}
	Observe that ${}^\sigma L_{\bar{F}}= k_\sigma L_{\bar{F}} k^{-1}_\sigma$ since both are the Levi subgroups of ${}^\sigma Q_{\bar{F}}$ containing $\bar{F}\otimes_F H$.
	
	Let $k_\sigma'$ be another candidate. Then $k^{-1}_\sigma k'_\sigma$ normalizes $Q_{\bar{F}}$. Since $Q_{\bar{F}}$ is self-normalizing (\cite[Proposition 1.2]{MR218364}), $k^{-1}_\sigma k'_\sigma$ lies in $Q_{\bar{F}}(\bar{F})\cap K(\bar{F})$. In view of the first paragraph, $k^{-1}_\sigma k'_\sigma$ also normalizes $L_{\bar{F}}$. This shows $k^{-1}_\sigma k'_\sigma\in L_{\bar{F}}(\bar{F})\cap K(\bar{F})$ (\cite[Proposition 1.6]{MR218364}). The independence as a character of $L_{\bar{F}}\cap (\bar{F}\otimes_F K)$ is now straightforward.
	
	It remains to prove the independence as a character of $\fl_{\bar{F}}$. More generally, we prove that every character of $\fl_{\bar{F}}$ is invariant under $L_{\bar{F}}(\bar{F})$-twists. Since $L_{\bar{F}}$ is connected, we may verify the invariance under $\fl_{\bar{F}}$-twists. This is clear.
\end{proof}

Let $\Gamma_{t_{Q_{\bar{F}}}}$ be the open subgroup of $\Gamma_F$ consisting of the elements $\sigma\in\Gamma_F$ such that ${}^\sigma Q_{\bar{F}}$ is $K(\bar{F})$-conjugate to $Q_{\bar{F}}$.
Let $\Gamma_{t_{Q_{\bar{F}}},\pi_{\bar{F}}}$ be the open subgroup of $\Gamma_F$ consisting of the elements $\sigma\in\Gamma_{t_{Q_{\bar{F}}}}$ such that $\sigma\ast\pi_{\bar{F}}\cong\pi_{\bar{F}}$. Set
$F(t_{Q_{\bar{F}}},\pi_{\bar{F}})\coloneqq \bar{F}^{\Gamma_{t_{Q_{\bar{F}}},\pi_{\bar{F}}}}$.

\begin{prop}\label{prop:rat_coh}
	\begin{enumerate}
		\item We have $\Gamma_{t_{Q_{\bar{F}}},\pi_{\bar{F}}}\subset \Gamma_{X_{\bar{F}}}$.
		\item The converse containment to (1) holds true if $\Gamma_{t_{Q_{\bar{F}}}}=\Gamma_F$,
	\end{enumerate}
\end{prop}

\begin{proof}
	Part (1) follows from \cite[Variant G]{MR3853058}. Part (2) is a consequence of the Beilinson--Bernstein equivalence.
\end{proof}

This gives an estimate of the field of rationality of $X_{\bar{F}}$:

\begin{cor}\label{cor:estimate}
	We have $F(X_{\bar{F}})\subset F(t_{Q_{\bar{F}}},\pi_{\bar{F}})$. This is equal if $V$ has regular highest weights.
\end{cor}

Finally, we discuss the Borel--Tits cocycle. Write $\lambda$ for the character of the complex linear algebraic group $L\cap (\bC\otimes_F K)$ corresponding to $\pi$. Let $\fp\subset\fg$ be the $-1$-eigenspace of the differential of $\theta$. Let $U$ be the unipotent radical of $Q$. Set $\lambda^\sharp=\lambda\otimes_\bC \wedge^{\dim_\bC(\mathfrak{u}\cap (\bC\otimes_F \fp))}(\mathfrak{u}\cap (\bC\otimes_F \fp))$.

Let $X_{\min}\subset X$ be the space of minimal $\bC\otimes_F K$-types of $X$. Explicitly, it is given by
\begin{equation}
	X_{\min}\cong \Ind^{\bC\otimes_F K}_{Q^-\cap (\bC\otimes_F K)}\lambda^\sharp\label{eq:min}
\end{equation}
(see \cite[Proposition 10.24]{MR1330919}).
We naturally have an $\bar{F}$-form $X_{\min,\bar{F}}$ by \cite[Chapter I, 3.5 (3)]{MR2015057}.

Assume that $X_{\min}$ is irreducible. This is satisfied if
\begin{enumerate}
	\item[1.] $Q\cap (\bC\otimes_F K)$ meets every component of $\bC\otimes_F K$, or
	\item[2.] $\langle\alpha^\vee,\lambda^\sharp\rangle>0$ for every $\bC\otimes_FH$-coroot $\alpha^\vee$ of $U$.
\end{enumerate}

\begin{ex}
	The condition 1 is satisfied for any $Q$ if $K$ is connected or
	$\bR\otimes_F G$ is isomorphic to $\GL_{2n+1}$ for $n\geq 0$. An easy counterexample is $G=\GL_2$ over $F=\bR$ with $K=\Oo(2)$ and $Q$ Borel.
\end{ex}

We are now able to compute the Borel--Tits cocycle of $X$ over $F(t_{Q_{\bar{F}}},\pi_{\bar{F}})$: Pick $F'$-forms $V_{F'}$, $Q_{F'}$, and $\pi_{F'}$ of $V_{\bar{F}}$, $Q_{\bar{F}}$, and $\pi_{\bar{F}}$ for a sufficiently large finite Galois extension $F'/F(t_{Q_{\bar{F}}},\pi_{\bar{F}})$ in $\bar{F}$. Moreover, we may assume that for each $\bar{\sigma}\in\Gamma_{F'/F(t_{Q_{\bar{F}}},\pi_{\bar{F}})}$, there exists $k_{\bar{\sigma}}\in K(F')$ such that ${}^{\bar{\sigma}} Q_{F'}= k_{\bar{\sigma}} Q_{F'}k_{\bar{\sigma}}^{-1}$.

\begin{prop}\label{prop:BT}
	\begin{enumerate}
		\item We have $\beta^{\BT}_{X_{\bar{F}}}=\beta^{\BT}_{X_{\min,\bar{F}}}$ in $H^2(\Gamma_{t_{Q_{\bar{F}}},\pi_{\bar{F}}},\bar{F}^\times)$. 
		\item The cohomology class $\beta^{\BT}_{X_{\min,\bar{F}}}\in H^2(\Gamma_{t_{Q_{\bar{F}}},\pi_{\bar{F}}},\bar{F}^\times)$ is the image of the 2-cocycle $(\lambda^\sharp_{F'}(k_{\bar{\sigma}\bar{\tau}}^{-1}\bar{\sigma}(k_{\bar{\tau}})k_{\bar{\sigma}}))\in H^2(\Gamma_{F'/F},(F')^\times)$.
	\end{enumerate}
\end{prop}

\begin{proof}
	Part (1) follows from Proposition \ref{prop:mult_one}. We can verify (2) in a similar way to \cite[Lemma 2.2.2]{MR4627704} or \cite[Theorem 5.46]{hayashisuper} (see also \cite{hayashierror} for the continuity argument).
\end{proof}

We can also relate $\beta^{\BT}_{X_{\min,\bar{F}}}$ directly with Borel--Tits' original cocycles in a certain connectivity assumption:
	
	\begin{prop}\label{prop:K^0}
		Suppose that $Q\cap (\bC\otimes_F K)$ meets every component of the complex linear algebraic group $\bC\otimes_F K$. Then $X_{\min,\bar{F}}$ is irreducible as an $\bar{F}\otimes_F K^0$-module.
		Moreover, the Borel--Tits cocycles of $X_{\min,\bar{F}}$ as representations of $\bar{F}\otimes_F K$ and $\bar{F}\otimes_F K^0$ are equal. 
	\end{prop}

\begin{proof}
	The first part is immediate from \eqref{eq:min}. The equality of the Borel--Tits cocycles then follows by the numerical description in Proposition \ref{prop:key_computation_BT} (also recall Construction \ref{cons:BT}).
\end{proof}

	\section*{Acknowledgments}
	
	The purely algebraic proof of the finite length result is due to Masatoshi Kitagawa. The author is indebted to him for his patient explanation of the proof.
	The author also thanks Yuki Yamamoto for helpful comments.

\end{document}